%% file: Hurwitzreell.tex
\documentclass [10pt,a4paper]{amsart}
\usepackage {hannah}
\usepackage {graphicx,psfrag,enumitem}
\usepackage[usenames,svgnames]{xcolor}
\usepackage{texdraw}
\usepackage [latin1]{inputenc}

\usepackage {amsmath,amssymb}
\usepackage{amsthm}
\usepackage[encapsulated]{CJK}
\usepackage{hyperref}
\usepackage{amscd}
\usepackage[shortalphabetic,nobysame]{amsrefs}

\hypersetup{
  colorlinks,
  pdftitle={The combinatorics of real double Hurwitz numbers with real positive branch points},
  pdfauthor={Mathieu Guay-Paquet, Hannah Markwig and Johannes Rau},
}

\usepackage{todonotes}

\graphicspath{{figs/}}

\DeclareMathOperator {\uc}{uc}

\DeclareMathOperator {\val}{val}

\DeclareMathOperator {\trop}{trop}

\DeclareMathOperator {\pari}{par}

\newcommand {\dunion}{\,\mbox {\raisebox{0.25ex}{$\cdot$} \kern-1.83ex $\cup$}
  \,}

\title[The combinatorics of real double Hurwitz numbers] {The combinatorics of real double Hurwitz numbers with real positive branch points}
\author{Mathieu Guay-Paquet, Hannah Markwig and Johannes Rau}
\address{Universit\"at des Saarlandes, Fachrichtung
  Mathematik, Postfach 151150, 66041 Saar\-br\"ucken, Germany }
\email{hannah@math.uni-sb.de, johannes.rau@math.uni-sb.de}
\address{%
LaCIM\\
Universit\'e du Qu\'ebec \`a Montr\'eal\\
201 Av du Pr\'esident-Kennedy\\
Montr\'eal QC\ \ H2X~3Y7\\
Canada}
\email{mathieu.guaypaquet@lacim.ca}

\begin{document}

\begin{abstract}
We investigate the combinatorics of real double Hurwitz numbers with real positive branch points using the symmetric group. Our main focus is twofold. First, we prove correspondence theorems relating these numbers to counts of tropical real covers and study the structure of real double Hurwitz numbers with the help of the tropical count. Second, we express the numbers as counts of paths in a subgraph of the Cayley graph of the symmetric group. By restricting to real double Hurwitz numbers with real positive branch points, we obtain a concise translation of the counting problem in terms of tuples of elements of the symmetric group that enables us to uncover the beautiful combinatorics of these numbers both in tropical geometry and in the Cayley graph.
\end{abstract}

 \maketitle

\section{Introduction}

\label{sec-intro}
In this paper, we study the combinatorics of real double Hurwitz numbers with only real and positive branch points.

(Complex) Hurwitz numbers count genus $g$ degree $d$ covers of a curve of genus $h$ with fixed ramification conditions at fixed points of the target. The \emph{ramification profile} of a point in the target, a partition of the degree $d$, encodes how many sheets of the map come together above this point.
We call a point a \emph{branch point} if its ramification profile is not $(1^d)$. If the profile is $(2,1^{d-2})$ we say that the branch point (resp.\ the ramification) is \emph{simple}. The Riemann-Hurwitz formula implies how many branch points we need to fix in order to obtain a finite count.
Such counts of covers date back to Hurwitz himself in the 19th century and have since then provided a fertile source for interesting problems connecting the geometry of covers, the moduli space of curves, the representation theory of the symmetric group and matrix models in probability theory. 

Double Hurwitz numbers are counts of covers of $\P^1$, where we fix two special ramification profiles $\mu$ and $\nu$ about $0$ and $\infty$ and only simple ramification elsewhere.
Double Hurwitz numbers feature a particularly rich structure investigated e.g.\ in \cites{gjv:ttgodhn, ssv:cbodhn, cjm:wcfdhn, Joh10}.
Tropical analogues of double Hurwitz numbers were introduced in \cite{CJM10} and have been a successful tool in obtaining structural results \cite{cjm:wcfdhn}.

In this paper, we study real covers of the projective line. In general, real algebraic geometry is much harder than complex algebraic geometry, that is, algebraic geometry over an algebraically closed field. This holds true also for real counts of covers. For example, counts of real covers may depend on the exact position of the branch points.
We restrict our attention to real double Hurwitz numbers with only real and positive branch points. For this situation, the count does not depend on the exact position of the chosen branch points. There is an ambiguity in the definition of such Hurwitz numbers: we can either count them with their real structure (we call these numbers $\tilde{H}_g(\mu,\nu)$) or without ($H_g(\mu,\nu)$).

By matching a cover with a monodromy representation, the count of a Hurwitz number is equivalent to choices of $n$-tuples of elements of $\mathbb{S}_d$ of fixed conjugacy class satisfying some conditions.
This holds also true for the real double Hurwitz numbers in question. We exploit the symmetric group approach to Hurwitz numbers to investigate their combinatorial properties in two directions:
\begin{enumerate}
 \item We study tropical real double Hurwitz numbers.
\item We investigate real double Hurwitz numbers in terms of paths in a subgraph of the Cayley graph of the symmetric group.
\end{enumerate}
We construct tropical real double Hurwitz numbers as weighted counts of graphs mapping to a line (i.e.\ tropical covers) which are colored in a way reflecting the real structure. We obtain correspondence theorems for the numbers $\tilde{H}_g(\mu,\nu)$ and $H_g(\mu,\nu)$. The tropical interpretation of real double Hurwitz numbers uncovers the relation between these numbers: an easy corollary of our correspondence theorems (corollary \ref{cor-tildeh}) implies that $\tilde{H}_g(\mu,\nu)=H_g(\mu,\nu)$ if $\mu$ and $\nu$ are not both in $\{d,(\frac{d}{2},\frac{d}{2})\}$. If $\mu$ and $\nu$ are both in $\{d,(\frac{d}{2},\frac{d}{2})\}$, their difference is also determined in corollary \ref{cor-tildeh}.

The study of correspondence theorems for Hurwitz numbers relating them to their tropical counterparts is not limited to an approach in terms of the symmetric group. For complex Hurwitz numbers, there is a general version using topological methods. General correspondence theorems for real Hurwitz numbers are the topic of a forthcoming paper of the second and third author.

By restricting to real double Hurwitz numbers with real positive branch points, we obtain a concise translation of the counting problem in terms of tuples of elements of the symmetric group, which enables us to express the combinatorics both in terms of tropical covers as well as in terms of paths in the Cayley graph in a useful way.

The paper is organized as follows.
In section \ref{sec-real} we introduce the counting problem.
In section \ref{sec-tropical}, we study tropical real double Hurwitz numbers. We prove correspondence theorems using an approach via the symmetric group.
In section \ref{sec-walls}, we study real double Hurwitz numbers as a map.
In section \ref{sec-cayley}, we translate some of these concepts to the Cayley graph.

\subsection{Acknowledgements}

The first author was supported by the Natural Sciences and Engineering Research Council of Canada (NSERC). The second author is supported by DFG-grant MA 4797/6-1 and GIF grant no. 1174-197.6/2011.

Part of this work was completed during the second and third author's stay at the Centre Interfacultaire Bernoulli at the \'{E}cole Polytechnique F\'{e}d\'{e}rale de Lausanne during the program ``Tropical geometry in its complex and symplectic aspects'' in 2014. The authors would like to thank the CIB for hospitality, and the CIB and NSF (National Science Foundation) for support.

We would like to thank Erwan Brugall\'{e}, Ilia Itenberg, Grigory Mikhalkin and Roland Speicher for helpful discussions. We thank an anonymous referee for useful comments on an earlier version.

\section{Real double Hurwitz numbers with positive real branch points}\label{sec-real}
We start by defining the relevant counts of real covers and discussing the relation to counts of tuples in the symmetric group satisfying certain properties.
Here, we restrict to only real and positive branch points. Under this requirement, the real Hurwitz numbers we discuss are invariant of the exact location of the branch points.
We introduce two notions of real Hurwitz numbers, either counting covers that allow a real structure, or counting covers together with the real structure.

\begin{definition}[Real double Hurwitz numbers]\label{def-Hurwitz}
Fix two partitions $\mu$ and $\nu$ of an integer $d\geq1$ and a genus $g$.
For a partition $\mu$, let $\ell(\mu)$ denote the number of parts of $\mu$.
Fix $r=2g-2+\ell(\mu)+\ell(\nu)$ positive points $p_1, \ldots, p_{r}$ on $\mathbb{P}_{\R}^1$, i.e.\ strictly between $0$ and $\infty$.

Let $\text{conj} : \P^1 \to \P^1$ be the involution given by conjugation, with fixed point locus $\mathbb{P}_{\R}^1$.

The \textit{real double Hurwitz number} $H_{g}(\mu,\nu)$ is defined as the weighted number of degree $d$ covers $\pi: C\rightarrow \P^1$ where
\begin{itemize}
\item $C$ is a smooth projective curve of genus $g$ over $\C$;
\item
the cover $\pi: C\rightarrow \P^1$ is real, i.e.\ there is a smooth involution $\varphi:C\rightarrow C$ satisfying $\pi \circ \varphi= \text{conj} \circ \pi$;
\item $\pi$ ramifies with profile $\mu$ over $0\in \P^1$ and with profile $\nu$ over $\infty\in \P^1$;
\item $\pi$ is simply ramified at $p_1,\ldots,p_{r}$.
\end{itemize}
As usual, we count such covers up to isomorphism, and each cover $\pi$ is weighted by $1/|\mbox{Aut}(\pi)|$.
Here, an isomorphism between two covers $\pi$ and $\pi'$ is an isomorphism of the corresponding curves $\alpha : C \to C'$ such that
$\pi = \pi' \circ \alpha$.
\end{definition}

Note that this number is independent of the locations of the positive $p_i\in \P^1_{\R}$.
This follows e.g.\ from lemma \ref{lem-tuples}.
 It follows from the Riemann-Hurwitz formula (see e.g.\ \cite{Har77}, Corollary IV.2.4) that $r$ is the number of simple ramifications we need to fix.
It also follows that $0$, $\infty$ and $p_1,\ldots,p_{r}$ are all branch points of the covers.
We chose not to put any markings on the ramification points.
For some purposes (cf.\ section \ref{sec-walls}), it is useful to mark the preimages of the special branch points $0$ and $\infty$ in order to make them distinguishable. However, the two definitions just differ by a factor of $|\Aut(\mu)| \cdot |\Aut(\nu)|$.

In a variant of this definition, we count covers with their real structures:
\begin{definition}[Real double Hurwitz numbers with real structures]
Let $d, \mu, \nu, g, p_1,\ldots, p_r$ be as in definition \ref{def-Hurwitz}. We set $\tilde{H}_{g}(\mu,\nu)$ the number of pairs $(\pi, \varphi)$ where $\pi:C\rightarrow \P^1$ is a cover satisfying the requirements of definition \ref{def-Hurwitz}, and $\varphi$ is a real structure,
i.e.\ a smooth involution $\varphi:C\rightarrow C$ satisfying $\pi \circ \varphi=\text{conj} \circ \pi$.
We count such pairs up to \emph{real} isomorphism, and each pair is weighted by $1/|\mbox{Aut}((\pi,\varphi))|$.
A \emph{real} isomorphism between two pairs $(\pi, \varphi)$ and $(\pi', \varphi')$ is an isomorphism $\alpha$ between $\pi$ and $\pi'$
satisfying also $\alpha \circ \varphi = \varphi' \circ \alpha$.
\end{definition}

By matching a cover with a tuple in the symmetric group encoding the monodromy and the involution, we obtain the following equivalent definition of real Hurwitz numbers \cite{Cad05}:
\begin{lemma}\label{lem-tuples}
 The real double Hurwitz number $\tilde{H}_g(\mu,\nu)$ (with $|\mu|=|\nu|=d$) equals 
$1/d!$
times the number of tuples $(\gamma,\sigma,\tau_1,\ldots,\tau_r)$ of elements of the symmetric group $\mathbb{S}_d$ satisfying
\begin{enumerate}
\item $\sigma$ has cycle type $\mu$;
\item the $\tau_i$ are transpositions;
\item $\tau_r\circ\cdots\circ\tau_1\circ\sigma$ has cycle type $\nu$;
\item the subgroup generated by $\sigma,\tau_1,\ldots,\tau_r$ acts transitively on the set $\{1,\ldots,d\}$;
\item \label{gammacond} $\gamma$ is an involution (i.e.\ $\gamma^2 = \id$) satisfying $$\gamma\circ\sigma\circ\gamma=\sigma^{-1}$$
and
 $$\gamma\circ (\tau_i\circ\cdots\circ\tau_1\circ\sigma)\circ\gamma= (\tau_i\circ\cdots\circ\tau_1\circ\sigma)^{-1}$$ for all $i=1,\ldots,r$.
\end{enumerate}

The real double Hurwitz number $H_g(\mu,\nu)$ equals 
$1/d!$
times the number of tuples $(\sigma,\tau_1,\ldots,\tau_r)$ satisfying the requirements above, and that there exists an involution $\gamma$ satisfying the above.
\end{lemma}

To prove this lemma, let us explain the meaning of $\gamma$ and condition \ref{gammacond}.

\begin{construction} \label{constr-realcover}
  We fix once and for all the following data: We pick $r$ real and strictly positive numbers/points $0 < p_1 < p_2 < \cdots < p_r < \infty$ on $\P^1_\R$.
	We choose $-1 \in \P^1$ as base point and fix non-intersecting paths $s_0, s_1, \ldots, s_r$ from $-1$ to $0, p_1, \ldots, p_r$ resp.\ which lie,
	except for the starting and end points, in the upper half of $\P^1$ (i.e.\ the half where the imaginary part is strictly positive). By ``adding'' small
	positively oriented loops around $0, p_1, \ldots, p_r$, we get loops $l_0, l_1, \ldots, l_r$ resp.\ which generate
	$\pi_1 := \pi_1(\P^1\setminus\{0,\infty,p_1,\ldots, p_r\}, -1)$. By our choices, we see that for all $i = 0, \ldots, r$
	\begin{equation}
		\label{eq:conjugatedloops}
		\text{conj} \circ (l_i \cdots l_0) = (l_i \cdots l_0)^{-1} \quad \in \pi_1.
	\end{equation}
	
	Given a tuple $(\gamma,\sigma,\tau_1,\ldots,\tau_r)$ satisfying the conditions of \ref{lem-tuples}, we construct a \emph{real} cover as follows.
	Using only $(\sigma,\tau_1,\ldots,\tau_r)$ and the first four conditions, we can use the well-known Hurwitz construction to obtain a (connected) cover
	$\pi : C \to \P^1$ with the given ramification profile $\mu$ resp.\ $\nu$ over $0$ resp.\ $\infty$ and simple ramification over each $p_i$.
	Moreover, the preimages of $-1$ are labeled, i.e.\ $\pi^{-1}(-1) = \{q_1, \ldots, q_d\}$, and the monodromy action of the $l_i$ is described by
	$(\sigma,\tau_1,\ldots,\tau_r)$.
	Based on $\gamma$, we define the real involution $\varphi$ as follows. Let $q \in C$ be an unramified point.
	Choose a path $s$ in $\P^1\setminus\{0,\infty,p_1,\ldots, p_r\}$ from $-1$ to $\pi(q)$, and let $t := \text{conj} \circ s$ be the conjugated path.
	Lift $s$ to a path $\tilde{s}$ with end point $q$ and let $q_k$ be its starting point.
	Lift $t$ to a path $\tilde{t}$ with starting point $q_{\gamma(k)}$ and let $q'$ be its end point. We set $\varphi(q) := q'$.
	First of all, note that this is well-defined, i.e.\ does not depend on the choice of $s$. To show this, it suffices to consider $q = q_l \in \pi^{-1}(-1)$
	and $s= l_i \cdot l_{i-1} \cdots l_0$ (as these loops also generate $\pi_1$).
	In this case $t = s^{-1}$ and the corresponding permutations are $\rho := \tau_i\circ\cdots\circ\tau_1\circ\sigma$ resp.\ $\rho^{-1}$.
	Our construction yields $\varphi(q_l) = q_{l'}$ with
	\[
	  l' = \rho^{-1}(\gamma(\rho^{-1} (l))).
	\]
	But condition \ref{gammacond} implies
	\[
	  \rho^{-1} \circ \gamma \circ \rho^{-1} = \gamma \circ \rho \circ \gamma \circ \gamma \circ \rho^{-1} =  \gamma
	\]
	and thus $\varphi(q_l) = q_{\gamma(l)}$ independent of $s$. By standard arguments, $\varphi$ can be extended to all points of $C$ and
	is a smooth involution. Moreover, the property $\pi \circ \varphi=\text{conj} \circ \pi$ holds by construction.
\end{construction}

\begin{proof}[Proof of \ref{lem-tuples}]
  The previous construction yields a map $h : \mathcal{T} \to \mathcal{R}$ from the set $\mathcal{T}$ of tuples $(\gamma,\sigma,\tau_1,\ldots,\tau_r)$
	satisfying the given condition to the set $\mathcal{R}$ of pairs $(\pi,\varphi)$ of covers $\pi$ with real structure $\varphi$, modulo
	real isomorphisms.
	Let us consider the action of $\mathbb{S}_d$ on $\mathcal{T}$ by conjugation (coordinatewise). It is well-known that the only change this
	action causes in the Hurwitz construction is a relabeling of $\pi^{-1}(-1)$. As we also conjugate $\gamma$ accordingly,
	we see that $h$ is invariant under this action. Given a pair $(\pi,\varphi)$ representing an element in $\mathcal{R}$,
	we can construct a preimage under $h$ as follows. We fix a labeling of $\pi^{-1}(-1)$ by $\{1, \ldots, n\}$ and
	define $(\sigma,\tau_1,\ldots,\tau_r)$ as the monodromy representation of the loops $l_0, l_1, \ldots, l_r$ (notation as above).
	Additionally, $\gamma$ is just given by the action of $\varphi$ on $\pi^{-1}(-1)$. If another pair $(\pi', \varphi')$ and
	a labeling of $\pi'^{-1}(-1)$ produce the same element in $\mathcal{T}$, the two pairs must be real isomorphic. Again,
	this statement is well-known when forgetting the real structures. Indeed, the isomorphism $\alpha$ between $\pi$ and $\pi'$ can be constructed
	for example by lifting paths (similar to the construction of $\varphi$ in \ref{constr-realcover}). Using this description
	and the fact that $\varphi$ and $\varphi'$ can also be described in term of lifting paths and $\gamma=\gamma'$, we see that $\alpha$
	is in fact a real isomorphism.
	It follows that
	\[
	  h : \mathcal{T}/\mathbb{S}_d \to \mathcal{R}
	\]
	is a bijection. Recall that an automorphism of a cover $\pi$ is determined by its action on any unramified fiber, i.e.\
	can be identified with a permutation. Under this identification, we see that for any tuple $T \in \mathcal{T}$ we have
	\[
	  \text{Stab}_{\mathbb{S}_d}(T) = \Aut_\R(h(T)).
	\]
	It follows that $|\mathcal{T}| / d!$ is equal to the \emph{weighted} cardinality of $\mathcal{R}$ and the statement follows.
\end{proof}


%

\section{Tropical real double Hurwitz numbers with positive real branch points}\label{sec-tropical}

\subsection{Tropical covers}
For our purposes, we can restrict to explicit abstract tropical curves, and to covers of a model of the tropical projective line containing two ends. We recall the relevant definitions for this situation.

A \emph{(abstract, explicit) tropical curve} is a connected metric graph $C$ satisfying the following properties. A vertex is called a \emph{leaf} if it is one-valent and an (inner) vertex otherwise. An edge $e$ is called an \emph{end} and has length $l(e)=\infty$ if it is adjacent to a leaf, otherwise it is called a \emph{bounded edge} and has a length $l(e)\in \R$. The \emph{valence} $\val(V)$ of each (inner) vertex is at least $3$. 
A \emph{flag} of a tropical curve is a tuple $(V,e)$ of a vertex $V$ and an adjacent edge $e$ that can be viewed as a directed edge $e$ pointing away from $V$.

The number $g=b^1(C)$, also known as the circuit rank of $C$, is called the \emph{genus} of the tropical curve $\Gamma$.

The \emph{combinatorial type} of a tropical curve is obtained by omitting the length data.

 We denote by $L$ the model of the tropical projective line with two ends, i.e.\ $L=\mathbb{R}\cup \{\pm \infty\}$.

\begin{definition}[Tropical covers] \label{def-tropcover}A \emph{tropical cover} of $L$ is a continuous map $\pi:C \rightarrow L$ from an abstract tropical curve $C$ satisfying:
\begin{itemize}
  \item $\pi$ is integral affine-linear on each edge $e$, i.e.\ if we understand
$e$ as open interval $(0,l(e))$ where $l(e)$ denotes its length, then
$\pi_{\mid e}$ maps $t\in (0,l(e))$ to $w_et+a$ for some starting point $a\in L$ and some nonzero integer $w_e$ which is defined up to sign and called the \emph{weight} of $e$.
If we fix a flag $(V,e)$ of an edge $e$ and denote by $V'$ the other vertex of $e$, we use the convention that $w_e$ is negative if $\pi(V')<\pi(V)$ and positive otherwise.
 \item $\pi$ fulfills the \emph{balancing condition} at each (inner) vertex $V$:
$$\sum_{(V,e)} w_e=0,$$ where we sum over all flags containing $V$ and use the sign convention from above.
 The sum of the positive weights (or equivalently, minus the sum of the negative weights) is called the \emph{local degree}
of $\pi$ at $V$ and is denoted by $\deg_{\pi}V$. For a point $a$ on an edge $e$ of $C$, we define the local degree to be equal to $\deg_{\pi}a=|w_e|$.
 \end{itemize}
\end{definition}

The continuity implies that leaves are mapped to $\pm\infty$.

Sometimes one also allows edges of weight $w_e=0$, i.e.\ edges which are contracted to a point.
As contracted edges do not play a role for counting covers, we neglect them
here.

\begin{definition}
 Let $\pi:C \rightarrow L$ be a cover. The balancing
condition implies that for every point $\tilde a$ in $L$ the sum
\begin{equation}
 \sum_{a|\pi(a)= \tilde a} \deg_\pi a
\end{equation}
is the same. This number is called the \emph{degree} $\deg(\pi)$ of $\pi$.
\end{definition}
\begin{definition}
 We say that a cover $\pi:C \rightarrow L$ is \emph{$3$-valent}, if $C$ has only $3$-valent vertices besides the leaves. We call the $3$-valent vertices the \emph{ramification points} of a cover and their images the \emph{branch points}.
\end{definition}

Two covers $\pi:C \rightarrow L$ and $\pi':C' \rightarrow L$ are called isomorphic, if there is an isomorphism $\varphi$ of the underlying abstract tropical curves (i.e.\ a homeomorphism respecting the edge lengths and the marking of the leaves) satisfying $\pi'\circ \varphi=\pi$.
\begin{remark}
 Note that the only automorphisms of a $3$-valent cover (where the images of the $3$-valent vertices are distinct) arise due to \emph{wieners} and \emph{balanced forks} as in \cite{CJM10}, Lemma 4.2 and Figure 2:
\begin{center}
\input{figs/wienerforks.TpX}
\end{center}
The automorphism group of a $3$-valent cover $\pi:C \rightarrow L$ thus has size $|\Aut(\pi)|=2^{W+B}$, where $W$ denotes the number of wieners and $B$ denotes the number of balanced forks.
\end{remark}

Note that the balancing condition implies that for an inner vertex $V$, either one of the three adjacent flags has negative weight and two positive weight or vice versa. We refer to an edge $e$ such that the flag $(V,e)$ has negative weight as an incoming edge and to an edge with positive weight as an outgoing edge. We encode this in pictures by drawing the corresponding arrows:
\begin{center}
 \input{./figs/arrows.pstex_t}
\end{center}


We refer to edges of even (resp.\ odd) weight as even edges (resp.\ odd edges).
\begin{definition}[Real tropical covers]\label{def-realcover}
 A $3$-valent tropical cover $\pi:C\rightarrow L$ together with a coloring of the edges of $C$ in three colors (that in this paper we encode with dashed lines, bold lines or normal lines) is called \emph{real} if the following conditions are satisfied:
\begin{itemize}
\item 
No edges except wiener or fork edges can be bold.
\item An odd edge which is not part of balanced fork or wiener is normal.
\item The colors of the three adjacent edges of any vertex have to fit one of the following pictures, depending on the parity of the weight and the orientation of the adjacent edges as above, respectively the reflection of the pictures with the arrows turned around if two edges are incoming and one is outgoing. If the parity is not specified it can be even or odd.
\begin{center}
\input{./figs/vertextypes.pstex_t}
\end{center}
\end{itemize}
\end{definition}

The \emph{combinatorial type} of a real tropical cover is the combinatorial type of the underlying abstract curve together with the colors, weights and directions for all edges (i.e.\ signed weights as in the convention of definition \ref{def-tropcover} for all flags).
The coloring essentially plays the role of a real involution in the classical picture (to be made precise in Lemma \ref{lem-gamma}).

Now we are ready to define tropical real double Hurwitz numbers given two partitions $\mu$ and $\nu$ of a degree $d$ and a genus $g$. In the following, we always assume that $\mu$, $\nu$, $d$ and $g$ are chosen such that $r=2g-2+\ell(\mu)+\ell(\nu)>0$, so that the tropical covers under consideration have at least one vertex.

\begin{definition}[Tropical real double Hurwitz numbers]\label{def-realtrophurwitz}
Fix two partitions $\mu$ and $\nu$ of an integer $d\geq1$ and a genus $g$. Fix $0<r=2g-2+\ell(\mu)+\ell(\nu)$ points $p_1, \ldots, p_{r}$ in $\R\subset L$. 
The \textit{tropical real double Hurwitz number (with real structures)} $\tilde{H}_{g}^{\trop}(\mu,\nu)$ is defined as the weighted number $\tilde{H}_{g}^{\trop}(\mu,\nu)=\sum_{\pi} \tilde{m}(\pi)$ of tropical degree $d$ real covers $\pi: C \rightarrow L$ where
\begin{itemize}
 \item $C$ is an abstract tropical curve of genus $g$;
\item the tuple of weights of ends adjacent to leaves mapping to $-\infty$ is $\mu$, the tuple of weights of ends adjacent to leaves mapping to $\infty$ is $\nu$;
\item the preimage $\pi^{-1}(p_i)$ contains a vertex of $C$.
\end{itemize}
It follows from the Euler characteristic of $C$ together with the Riemann-Hurwitz formula that $C$ has only $3$-valent inner vertices, and that there are $r>0$ such vertices.

We now define the multiplicity $\tilde{m}(\pi)$ with which a real tropical cover $\pi:C\rightarrow L$ contributes to $\tilde{H}_{g}^{\trop}(\mu,\nu)$.
Assume $\pi$ has $W$ wieners, $B$ balanced forks, $E$ dashed or normal even bounded edges and $k$ bold wieners of weights $w_1,\ldots,w_k$. Then we define $$\tilde{m}(\pi):=\frac{1}{2^{W+B}}\cdot 2^{E}\cdot \prod_{i=1}^k w_i.$$

The  \textit{tropical real double Hurwitz number (without real structures)} $H_{g}^{\trop}(\mu,\nu)$ is defined analogously, we only change the multiplicity with which we count, i.e.\ we set $H_{g}^{\trop}(\mu,\nu)=\sum_{\pi} m(\pi)$. Here, $m(\pi)= \tilde{m}(\pi)$ except for the following two situations that we refer to as \emph{chains of wieners}. Chains of wieners can only appear if $\mu,\nu\in\{d,(\frac{d}{2},\frac{d}{2})\}$.
\begin{enumerate}
 \item If $d\equiv 2 \;\mod 4$ there is a chain of wieners with the following colors whose multiplicity $m(\pi)$ we define to be $0$:
\vspace{1ex}
\begin{center}
 \input{./figs/wienerchain1.pstex_t}
\end{center}
\vspace{1ex}
\item If $d\equiv 0\; \mod 4$ any balanced fork or wiener in a chain of wieners can be either bold or normal, e.g.\ as in the picture:
\vspace{1ex}
\begin{center}
 \input{./figs/wienerchain2.pstex_t}
\end{center}
\vspace{1ex}
We say that two such chains of wiener $\Gamma_1$ and $\Gamma_2$ have \emph{complementary colors}, if each wiener or balanced fork which is bold in $\Gamma_1$ is normal in $\Gamma_2$ and vice versa.
We define the sum of the multiplicities of a pair of chains of wieners in complementary colors to be
$$ m(\Gamma_1)+m(\Gamma_2):= \tilde{m}(\Gamma_1)+\tilde{m}(\Gamma_2)-2^{W-1},$$
where $W$ denotes the number of wieners.
\end{enumerate}

\end{definition}
It turns out that chains of wieners are the only covers that allow more than one coloring satisfying the requirements of definition \ref{def-realcover} (see lemma \ref{lem-wienerchain}).

Obviously, the degree of a cover contributing to $\tilde{H}_{g}^{\trop}(\mu,\nu)$ or $H_{g}^{\trop}(\mu,\nu)$ is $d=|\mu|=|\nu|$.

\subsection{Monodromy graphs}
Counting tropical covers as above can be simplified by grouping covers with the same combinatorial type, and neglecting edge lengths.
The result obtained from neglecting edge lengths is what we call a monodromy graph. We use the same notions such as leaf, inner vertex etc.\ for monodromy graphs as for abstract tropical curves.

\begin{definition}\label{def-monodromygraph}
For fixed $g$ and partitions $\mu$ and $\nu$, a graph $\Gamma$ is a \emph{(real) monodromy graph} if:
\begin{enumerate}
\item $\Gamma$ is a connected, directed graph of genus $g$.
\item All other vertices of $\Gamma$ are $3$-valent.
\item Every edge $e$ of the graph is equipped with a \emph{weight} $w(e)\in \N$. The weights  satisfy the \emph{balancing condition} at each inner vertex: the sum of all weights of  incoming  edges equals the sum of the weights of all outgoing edges.
\item $\Gamma$ has $\ell(\mu)+\ell(\nu)$ leaves, $\ell(\mu)$ of them oriented inwards
(the \emph{in-ends}), and $\ell(\nu)$ of them oriented outwards (the \emph{out-ends}).
The partition given by the weights of all in-ends resp.\ out-ends is $\mu$ resp.\ $\nu$.
\item The inner vertices are ordered compatibly with the partial ordering induced by the directions of the edges.
\item The edges of $\Gamma$ are colored following the rules of definition \ref{def-realcover}.
\end{enumerate}
\end{definition}

From a cover $\pi:C\rightarrow L$, we obtain a monodromy graph $\Gamma$ by taking the combinatorial type of $C$, ordering the inner vertices according to their images in $L$ and orienting the edges as follows: if $e$ is an edge connecting the vertices $V'$ and $V$ and $\pi(V')<\pi(V)$, we let $e$ point from $V'$ to $V$.
Vice versa, from a monodromy graph we obtain a tropical cover of $L$ such that the preimages of the $p_i$ each contain a $3$-valent vertex by mapping the vertices as prescribed by the ordering. The images of the edges are then determined, and since the weights together with the image intervals $[p_i,p_j]\subset \R\subset L$ also determine the length in the abstract tropical curve, the tropical cover is uniquely recovered from the monodromy graph.

\begin{remark}\label{rem-dircyc}
Since the vertices in a monodromy graph are totally ordered, any orientation occurring has no directed cycles. It follows from the balancing condition that $\Gamma$ cannot have sinks or sources at inner vertices.
\end{remark}

\begin{remark}
We can define the multiplicity $ \tilde{m}(\Gamma)$ of a monodromy graph $\Gamma$ analogously to definition \ref{def-realtrophurwitz}.
Obviously, the real tropical Hurwitz number equals
\begin{equation} \tilde{H}_{g}^{\trop}(\mu,\nu)=\sum_\Gamma \tilde{m}(\Gamma),\label{eq-mongraphs}\end{equation}
where the sum goes over all real monodromy graphs for $g$, $\mu$ and $\nu$.
\end{remark}

\begin{remark}\label{rem-moncount}
 We can simplify the count of tropical covers further by grouping covers with the same combinatorial type, i.e.\ by declaring two monodromy graphs equivalent if they differ at most in the total ordering of the vertices. Obviously, each representative in an equivalence class is counted with the same weight, and so we can sum over equivalence classes instead and weight each class by $ \tilde{m}(\Gamma)$ times the number of total vertex orderings respecting the partial ordering induced by the orientation of the edges.
Abusing notation, we denote by $\Gamma$ the data defining an equivalence class, i.e.\ a monodromy graph without an order of the vertices.
For an equivalence class $\Gamma$, we denote its size (i.e.\ the number of total vertex orderings respecting the partial ordering induced by the orientation of the edges) by $o(\Gamma)$. The discussion above can be summed up by saying that
$$ \tilde{H}_g^{\trop}(\mu,\nu)= \sum_{\Gamma} o(\Gamma) \tilde{m}(\Gamma)$$
where the sum goes over all monodromy graphs without vertex orderings.
The analogous statement holds for $H_g^{\trop}(\mu,\nu)$.
\end{remark}

\begin{example}
 In Figure \ref{ex-422}, we compute $\tilde{H}_{1}^{\trop}((4),(2,2))= 2+1+1+2+1+1=8$ using monodromy graphs.
When we want to determine $H_{1}^{\trop}((4),(2,2))$, we have to subtract $2^{W-1}=1$ for each of the two pairs of chains of wieners of complementary colors, so we obtain $H_{1}^{\trop}((4),(2,2))= 8-2=6$.
\begin{figure} 
 \input{./figs/h1422.pstex_t}
 \caption{$\tilde{H}_{1}^{\trop}((4),(2,2))= 2+1+1+2+1+1=8$}
 \label{ex-422}
\end{figure}
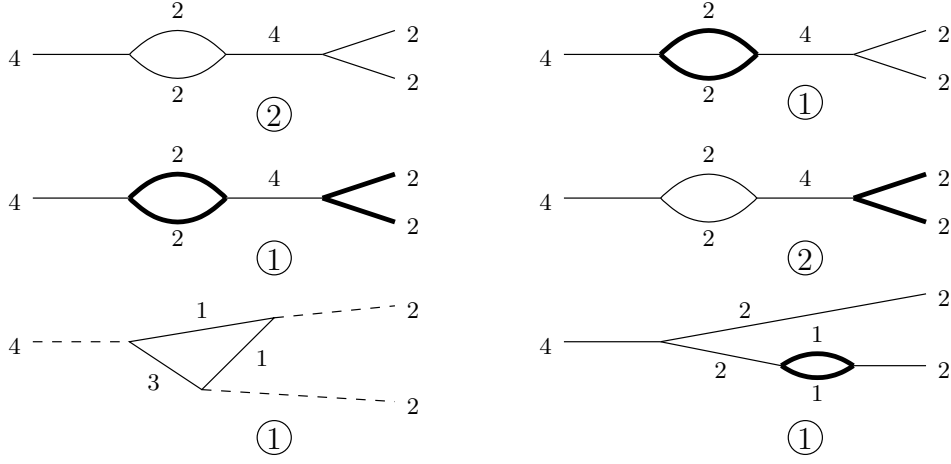
\end{example}

\subsection{A real monodromy graph from a tuple in \texorpdfstring{$\mathbb{S}_d$}{Sd}}\label{sec:mon-graph}

Similarly to \cite{CJM10}, Lemma 4.2, we can construct a monodromy graph from a tuple $(\gamma,\sigma,\tau_1,\ldots,\tau_r)$ of elements of $\mathbb{S}_d$ satisfying the requirements of lemma \ref{lem-tuples}:

We start by drawing edges with weights as given by the cycle lengths of the permutation $\sigma$. A transposition can either cut or join two cycles, so we introduce a vertex that either cuts one edge into two or joins two edges to one.
We encode the effect of the conjugation with $\gamma$ in the colors.

Let us first discuss the effect of the conjugation.
\begin{lemma}\label{lem-gamma}
 Let $\sigma,\gamma \in \mathbb{S}_d$ satisfy $\gamma \circ \sigma\circ\gamma=\sigma^{-1}$ and $\gamma^2=\id$. Let $\sigma=c_1\circ \cdots\circ c_r$ be the cycle decomposition of $\sigma$.
Then the conjugation with $\gamma$
\begin{itemize}
 \item can exchange two cycles of the same length, i.e.\ satisfy $\gamma\circ c_i \circ \gamma= c_j^{-1}$ for $i\neq j$, then also $\gamma\circ c_j \circ \gamma= c_i^{-1}$ and all numbers in the support of $c_i, c_j$ are in the support of $\gamma$.
\item or invert a cycle $c_i$. Then if $c_i$ is of odd length, all numbers but one of the support of $c_i$ appear in the support of $\gamma$. If $c_i$ is of even length, either all or all but two numbers of the support of $c_i$ appear in the support of $\gamma$. In the latter case, the two fixed numbers are of distance $\frac{r}{2}$ apart, where $r$ denotes the length of the cycle. In the former case, there are two pairs of neighboring numbers which are exchanged by the action of $\gamma$, and these pairs are again of distance $\frac{r}{2}$ apart.
\end{itemize}
\end{lemma}
The following picture shows an example of the action of $\gamma$ exchanging two cycles, acting on an odd cycle with one fixed point, acting on an even cycle without fixed point and with two fixed points:
\begin{center}
 \input{./figs/arrowsgamma.pstex_t}
\end{center}

\begin{proof}
 If $c_i=(m_1,...,m_r)$ then the conjugation satisfies $\gamma\circ c_i \circ \gamma= (\gamma(m_1),\ldots,\gamma(m_r))$. If $\gamma\circ c_i \circ \gamma= c_j^{-1}$, then all numbers in the support of $c_i$ must be thrown to numbers in the support of $c_j$. Since $\gamma$ is a product of disjoint transpositions, we also have $\gamma\circ c_j \circ \gamma= c_i^{-1}$.
Now assume  $\gamma\circ c_i \circ \gamma= c_i^{-1}$. Since $\gamma$ is a product of disjoint transpositions, the size of its support is even. Assume first $c_i$ is of odd length, then the support of $c_i$ minus the support of $\gamma$ has odd size. It cannot consist of more than one element, since otherwise the action of $\gamma$ could not invert $c_i$. If $c_i$ is of even length, its support minus the support of $\gamma$ has even length. Again, it cannot consist of more than two elements, since otherwise $c_i$ could not be inverted. If there are two fixed numbers, then they need to be of distance $\frac{r}{2}$ apart. If there are no fixed numbers, we can relabel the numbers so that the cycle is $(1\;\ldots\; r)$ and that $\gamma$ is $(1\; r) (2 \; r-1) \cdots (\frac{r}{2}\; \frac{r}{2}+1)$. Then the two pairs of neighbors are $(1,r)$ and $(\frac{r}{2}, \frac{r}{2}+1)$. They are $\frac{r}{2}$ positions apart in the cycle.
\end{proof}

\begin{construction}\label{const-monodromy}
 Given a tuple $(\gamma,\sigma,\tau_1,\ldots,\tau_r)$ of elements of $\mathbb{S}_d$ satisfying the requirements of lemma \ref{lem-tuples}, draw left ends of weights the cycle lengths of $\sigma$.
Denote $\sigma_i:=\tau_i\circ\cdots\circ\tau_1\circ\sigma$. Compared to $\sigma_{i-1}$, we either cut a cycle or join two cycles, so we either draw a 3-valent vertex with two edges pointing out, or we merge two edges to one out-edge with a three-valent vertex.
We encode the effect of $\gamma$ in colors:
if the action of $\gamma$ exchanges two cycles, draw the corresponding edges bold. If $\gamma$ has two fixed elements in an even cycle, draw the corresponding edge dashed. All other edges are drawn with normal lines.
\end{construction}

\begin{lemma}
 For a tuple $(\gamma,\sigma,\tau_1,\ldots,\tau_r)$ of elements of $\mathbb{S}_d$ satisfying the requirements of lemma \ref{lem-tuples}, construction \ref{const-monodromy} yields a real monodromy graph. In particular, the coloring satisfies the rules of definition \ref{def-realcover}.
\end{lemma}
\begin{proof}
First note that cycles which are exchanged by the action of $\gamma$ must correspond to edges in a balanced fork or wiener. They can only appear from and be merged to one even normal edge; any other development does not satisfy the requirement on the action of $\gamma$.
If an even edge cuts into two odd edges, then the even cycle must have two fixed numbers since both odd cycles have one.
If an odd edge cuts into an odd and an even, then the even cycle cannot have fixed numbers since both odd cycles have one. Assume an even edge of weight $r$ cuts into two even edges of weight $r_1$ and $r_2$. If one of the new even edges, without loss of generality the one of weight $r_1$, had two fixed numbers, they would be $\frac{r_1}{2}$ numbers apart in the corresponding cycle. But then the fixed numbers cannot be $\frac{r}{2}$ apart in the cycle corresponding to the edge of weight $r$. Thus all involved cycles cannot have any fixed numbers and all lines are normal. The analogous reasons hold for the situation where two edges are merged.
In total, we can get exactly the pictures which are admitted for real monodromy graphs.
\end{proof}

Vice versa, for each monodromy graph, we can find a tuple yielding this graph when applying construction \ref{const-monodromy}. This will be obvious as soon as we find a formula for the number of tuples that give the same monodromy graph without markings on the ends and observe that it is nonzero. We have to understand this number anyway to obtain a weighted bijection.
This is the content of the following subsection.

\subsection{Real cut and join relations}

As usual, when hitting a permutation with a transposition, either a cycle is cut into two or two cycles are joined to one, and we can count how many transpositions there are that take us from one cycle type to another. Here, we want to take into account the effect of the involution $\gamma$ as well: we only count transpositions $\tau$ that satisfy $\gamma \tau \sigma \gamma= (\tau \sigma)^{-1}$.

\begin{lemma}\label{lem-vertexmult}
 The cut and join multiplicities, counting the number of transpositions that produce a certain cycle type from a given cycle type, each respecting the action of a given involution, are as given in the following picture:
\begin{center}
 \input{./figs/vertexmult.pstex_t}
\end{center}
The numbers in brackets are used if the two edges are indistinguishable, i.e.\ belong to a wiener or a balanced fork.
\end{lemma}

\begin{proof}
 Consider a vertex cutting an even edge into two bold edges. Without loss of generality, the cycle corresponding to the even edge is $(1\; \ldots\; 2k)$, and the involution is $(1 \;2k)(2 \;2k-1) \cdots (k\; k+1)$. How many transpositions cut $(1\;\ldots\; 2k)$ into two $k$-cycles in a way that $\gamma$ swaps them? One of the cycles must contain the numbers $1,\ldots,k$ and the other $k+1,\ldots,2k$, so there is only one way to cut, namely with $(1\; k+1)$.

Now consider an even dashed edge being cut into two odd edges of weights $k_1$ and $k_2$. We need to cut from the $k_1+k_2$-cycle corresponding to the even edge $k_1$ numbers around one of the two fixed numbers, so there are two possibilities to choose.
If $k_1=k_2$ and
the two edges are indistinguishable,
there is only one possibility since we do not distinguish the two cycles.

Let two odd edges be merged to one even dashed edge. Assume the cycles corresponding to the two odd edges are written in the order such that the middle number is the fixed number with respect to $\gamma$, then we have to choose the transposition consisting of the two starting numbers to merge.

Now assume two bold edges of weights $k$ are merged to an even edge. Without loss of generality, let the first $k$-cycle by $(1\;\ldots\; k)$, the second $(k+1\;\ldots\; 2k)$ and the involution $(1 \;2k) (2\; 2k-1)\cdots (k\; k+1)$.
We can then use the transpositions $(1 \;k+1)$ and $(i\; 2k+2-i)$ for $i=2,\ldots, k$ to merge, so there are $k$ choices.

Next assume we have an odd edge of weight $k$ cut into an even and an odd edge of weight $k_1$ and $k_2$. We need to cut out $k_2$ numbers around the fixed number in the $k$-cycle, so there is one transposition to choose.

If we have an even edge of weight $k$ cut into two even edges of weights $k_1$ and $k_2$, we can assume without loss of generality that the $k$-cycle is $(1\;\ldots\; k)$ and the involution is $(1\;k)(2\; k-1) \cdots (\frac{k}{2}\;\frac{k}{2}+1)$. That is, the involution contains two transpositions as factors that exchange neighboring numbers in the $k$-cycle, namely $(1\;k)$ and $(\frac{k}{2}\;\frac{k}{2}+1)$. When cutting, we can choose which pair of neighboring numbers exchanged by $\gamma$ should go into which cycle. If $k_1=k_2$ and
the two edges are indistinguishable,
there is only one possibility since we do not distinguish the two cycles.

Assume two even edges are merged into an even edge. As before, we have four possibilities to write the two even cycles with the exchanged neighbors outside or in the middle (two for each cycle). In each of the four cases, we can choose the transposition consisting of the two starting numbers, so altogether we have 4 choices.

Finally, let an odd and an even edge be merged into an odd edge. We assume again that the odd cycle is written with the fixed number in the middle, and the even cycle with one exchanged neighboring pair at the outsides and one in the middle (there are two possibilities to choose which neighboring pair should go where). Then we can pick the two starting numbers to join. Altogether, we have a choice of two transpositions.
\end{proof}

\subsection{The correspondence theorem}
Now we only need to combine the previous results to obtain the correspondence theorem for real Hurwitz numbers.

\begin{lemma}\label{lem-mult}
Fix a monodromy graph $\Gamma$. Assume it has $k$ bold wieners of weights $w_1,\ldots,w_k$. Denote by $W$ the total number of wieners, by $B$ the number of balanced forks and by $E$ the number of even dashed or normal bounded edges.
Then there are  $d!\cdot \frac{1}{2^{W+B}}\cdot 2^{E}\cdot \prod_{i=1}^k w_i$ tuples that yield $\Gamma$ when applying construction \ref{const-monodromy}.
\end{lemma}
\begin{proof}
 We only need to count how many possibilities we have to label a given monodromy graph with elements of the symmetric group. For the left ends, we can pick any permutation $\sigma$ having cycle type $\mu$; there are $\frac{d!}{\mu_1\cdots \mu_l \cdot |\Aut(\mu)|}$ such permutations. We have $\frac{|\Aut(\mu)|}{2^{B_l}}$ possibilities to choose which cycle to attach to which end, where $B_l$ denotes the number of left balanced forks.
Once we fix $\sigma$, how many possibilities are there to choose $\gamma$ satisfying $\gamma \circ \sigma \circ \gamma= \sigma^{-1}$? Of course, this number depends on $\mu$ only and not on the choice of $\sigma$.
For an odd cycle of length $\mu_i$, we have $\mu_i$ possibilities to choose the fixed number of the action of $\gamma$, by lemma \ref{lem-gamma}.
For an even cycle corresponding to a dashed or normal edge, we either have to pick two fixed numbers which are $\frac{\mu_j}{2}$ apart, or we have to pick two pairs of neighbors being exchanged which are again $\frac{\mu_j}{2}$ apart. In any case, the choice of one fixed number (resp.\ pair of neighbors) determines the other, so there are $\frac{\mu_j}{2}$ choices.
For a pair of bold ends of weight $\mu_k$, assume without loss of generality the two corresponding cycles are $(1\;\ldots\; \mu_k)$ and $(\mu_k+1\;\ldots\; 2\mu_k)$. Once we fix a number to be exchanged with $1$ in $\gamma$ --- for simplicity, let us here take $2\mu_k$ --- all other transpositions of $\gamma$ are also fixed by the condition. Here, we would get $\gamma=(1\; 2\mu_k)(2\; 2\mu_k-1) \cdots (\mu_k\; \mu_k+1)$. Thus we have $\mu_k$ choices for the pair.

Assume we have $s$ pairs of left bold ends, without loss of generality of weights $\mu_1,\ldots, \mu_{s}$, and assume we have $E_l$ dashed or normal even left ends.
If we divide the total number of choices of $\gamma$ by the product of all parts $\mu_1 \cdots \mu_l$, we get $\frac{1}{2}^{E_l}\cdot \frac{1}{\mu_1\cdots \mu_s}$.

For any vertex, we have to multiply by the numbers given in lemma \ref{lem-vertexmult}.
Note that for every vertex joining a pair of two bold ends, we obtain a corresponding factor of $\mu_i$, $1\leq i\leq s$, which cancels with the second factor above.

For some vertices, the multiplicities depend on the question whether the two outgoing edges are indistinguishable or not. They are indistinguishable if and only if they belong to a wiener or a balanced right fork.
We can thus multiply by $\frac{1}{2}^{W+B_r}$, where $B_r$ denotes the number of balanced right forks, and then add a factor of $2$ for any such vertex in exchange.
Note that after imposing this convention, we obtain a factor of $2$ for any dashed or normal even edge that is incoming for a vertex. For any dashed or normal even left end, these factor cancel with the $\frac{1}{2}^{E_l}$ from above, we are left with factors of $2$ for each bounded dashed or normal even edge. In addition, we have factors of $w_i$ for every wiener.
The statement follows.
\end{proof}

\begin{theorem}[Correspondence Theorem for real double Hurwitz numbers with real positive branch points, counted with real structure]\label{thm-corres}
 Algebro-geometric and tropical real double Hurwitz numbers with real positive branch points coincide, i.e.\ we have $$\tilde{H}_g(\mu,\nu)=\tilde{H}_g^{\trop}(\mu,\nu).$$
\end{theorem}
\begin{proof}
 By equation (\ref{eq-mongraphs}), we have $\tilde{H}_g^{\trop}(\mu,\nu)= \sum_\Gamma \tilde{m}(\Gamma)$ where the sum goes over all monodromy graphs. 
Denote the number of tuples that yield a given monodromy graph $\Gamma$ by $t_\Gamma$. Note that by lemma \ref{lem-mult} we have $t_\Gamma = d! \cdot \tilde{m}(\Gamma)$, so we get
\[
  \tilde{H}_g^{\trop}(\mu,\nu)=\frac{1}{d!} \sum_\Gamma t_\Gamma= \frac{t}{d!},
\]
where $t$ denotes the total number of tuples. By lemma \ref{lem-tuples}, the latter equals $\tilde{H}_g(\mu,\nu)$.
%
\end{proof}

The following lemmas are preparations for the correspondence theorem for real double Hurwitz numbers (not counted with real structures).
We now have to count tuples $(\sigma, \tau_1,\ldots,\tau_r)$ such that there exists an involution $\gamma$ satisfying the requirements of lemma \ref{lem-tuples}. We still use construction \ref{const-monodromy} to relate tuples to monodromy graphs,  but now we need to pay attention to tuples that only differ in the involution.

\begin{lemma}\label{lem-gammaunique}
 Fix a tuple $(\sigma,\tau_1,\ldots,\tau_r)$ and a suitable involution $\gamma$. Let $\Gamma$ be the monodromy graph obtained from this data as in construction \ref{const-monodromy}. Then there exists no other involution $\gamma'$ that together with $(\sigma,\tau_1,\ldots,\tau_r)$ yields $\Gamma$ (including the coloring).
\end{lemma}
\begin{proof}
 We have seen in the proof of lemma \ref{lem-mult} how many involutions we can choose for a given $\sigma$: for each odd cycle of length $k$, we have $k$ choices for the fixed number; for an even cycle corresponding to a dashed or normal edge we have $\frac{k}{2}$ choices of either a pair of fixed numbers or two pairs of exchanged neighbors; for a pair of bold edges of weight $k$, we have $k$ choices, given by fixing an image of one number in one of the cycles.
For any left end, consider its adjacent vertex, and assume it is the $i$th in the ordering of the vertices. The other edges adjacent to this vertex correspond in a unique way to cycles of $\tau_i\circ \cdots\circ \tau_1\circ \sigma$. Assume first the vertex cuts an odd edge into an odd and an even edge. Any possible involution $\gamma'$ needs to fix one number of the longer odd cycle, and accordingly the same number in the cut odd cycle. It needs to exchange two pairs of neighbors of the even cycle.
We assume without loss of generality that the odd cycle is $(1\;\ldots\; k)$ and that the transposition $\tau_i$ is $(1\; l)$ for some $l$. (Note that the other transpositions have no effect on this cycle, since the $i$th vertex is adjacent to this end.) Then the two cycles after cutting are $(1\;\ldots\; l-1)$ and $(l\;\ldots\; k)$. Assume without loss of generality that $(1\;\ldots\; l-1)$ is of odd length. The involution $\gamma'$ exchanges only two neighbors in $(1\;\ldots\; l-1)$, so the other pair of neighbors in $(l\;\ldots\; k)$ must be obtained from the cutting, i.e.\ they are $l$ and $k$. With that, the action of the involution both on $(l\;\ldots\; k)$ and on $(1\;\ldots\; k)$ is determined.
Assume now the vertex cuts an even edge into two bold edges. Without loss of generality, we can assume the even cycle is $(1\;\ldots\; 2k)$ and the transposition is $(1\;k+1)$, cutting it into $(1\;\ldots\; k)$ and $(k+1\;\ldots\; 2k)$. The involution $\gamma'$ needs to exchange two pairs of neighbors of $(1\;\ldots\; 2k)$ but cannot exchange any neighbors of $(1\;\ldots\; k)$ and $(k+1\;\ldots\; 2k)$. It follows that the two pairs of neighbors are $1$ and $2k$ and $k$ and $k+1$ and with that, $\gamma'=\gamma$ is fixed.
Let the vertex cut a dashed even edge into two odd edges. In each odd cycle, the involution exchanges two neighbors. In the even cycle, no neighbor is exchanged. So the neighbors have to appear in the cut. With that, the involution on the odd cycles and accordingly, also on the even cycle is fixed.
Assume the vertex cuts a normal even edge into two normal even edges. The involution exchanges two pairs of neighbors in each small even cycle, so altogether 4 pairs, but only two pairs in the big cycle. So, again, two such pairs must appear when cutting and $\gamma'$ is fixed.
Analogous arguments show that for a vertex joining two edges, the involved cycles fix the involution on the involved numbers. Since every left end is adjacent to a vertex, the involution on $\{1,\ldots,d\}$ is fixed. It follows that for a fixed monodromy graph and suitable tuple $(\sigma, \tau_1,\ldots,\tau_r)$ there is exactly one involution that satisfies the requirements.
\end{proof}

\begin{lemma}\label{lem-wienerchain}
Given a tuple $(\sigma,\tau_1,\ldots,\tau_r)$ and a suitable involution $\gamma$, if there is another involution $\gamma'\neq \gamma$ satisfying the requirements for $(\sigma,\tau_1,\ldots,\tau_r)$, then the monodromy graph $\Gamma$ corresponding to $(\sigma,\tau_1,\ldots,\tau_r)$ and $\gamma$ is a chain of wieners (see definition \ref{def-realtrophurwitz}).
\end{lemma}
\begin{proof}
First note that the monodromy graph associated to  $(\sigma,\tau_1,\ldots,\tau_r)$ and $\gamma'$ differs from $\Gamma$ only in the colors. We therefore first assume we are given a graph without colors and discuss possibilities to color the edges.
Obviously, we can a priori color even wieners normal or bold.
If the graph contains an odd edge which is not part of a wiener or balanced fork, then its color has to be normal. All colors of adjacent edges are also determined, and since the graph is connected, the color for all edges, except for even wieners, is determined.
Assume that all the odd edges of the graph are part of a wiener or balanced fork, and pick such an odd edge.
It has an adjacent even edge. If this even edge evolves into something different --- is either cut into two cycles of different lengths, or joined to another (necessarily even) edge --- its color is determined and with that also the color of the odd wiener or balanced fork. We only have two options for coloring if the even edge evolves into a chain of wieners: either we color each even edge normal and each odd edge bold, or each odd edge normal and each even edge dashed.
If the graph has no odd edge, then the color of each edge, except wieners which can be normal or bold, is determined.
Now assume we have a graph with even wieners but which is not a chain of wieners. It contains a vertex for which the colors of the adjacent edges are determined.
We have seen in the proof of lemma \ref{lem-gammaunique} that for any such vertex the involution $\gamma$ is determined by the tuple $(\sigma,\tau_1,\ldots,\tau_r)$. Since $\Gamma$ is connected, the only possibility to get different involutions $\gamma'$ arises when $\Gamma$ is a chain of wieners.
\end{proof}

 \begin{lemma}\label{lem-oddwienerchain}
Let $d\equiv 2 \;\mod 4$, $\mu,\nu\in\{d, (\frac{d}{2}\frac{d}{2})\} $ and $(\sigma,\tau_1,\ldots,\tau_r)$ a tuple with suitable involution $\gamma$ such that the monodromy graph for the pair is a chain of wieners as in definition \ref{def-realtrophurwitz}(a).
Then there exists $\gamma'$ satisfying the requirements such that in the monodromy graph, every edge that used to be dashed is now normal and every edge that used to be normal is now bold.
 \end{lemma}
\begin{proof}
Denote $k=\frac{d}{2}$. Without loss of generality, we can assume that $$\sigma=(1\;\ldots\; 2k)\mbox{  or  }(1 \;\ldots\; k) (k+1 \;\ldots\; 2k), \mbox{  and that}$$  $$\gamma= (1\; k) (2\; k-1) \;\ldots\; (\frac{k-1}{2}\;\frac{k+1}{2}) (k+1\; 2k)(k+2\; 2k-1)\cdots (\frac{2k-1}{2}\;\frac{2k+1}{2}).$$ It follows that $\tau_1=\cdots =\tau_r=(1\; k+1)$. In fact, the multiplicity of this graph is one (or two if there are two balanced forks) and this is the only tuple we have to consider. If we now set $\gamma'= (1\; 2k)(2 \; 2k-1)\cdots (k\; k+1)$ then the tuple together with $\gamma'$ yields the graph with every edge that used to be dashed normal and every edge that used to be normal bold.
\end{proof}

\begin{lemma}\label{lem-complementary}
Let $d\equiv 0 \;\mod 4$, $\mu,\nu\in\{d, (\frac{d}{2}\frac{d}{2})\} $. Let $(\sigma,\tau_1,\ldots,\tau_r)$ be a tuple and $\gamma$ a suitable involution, such that the monodromy graph for the pair is a chain of wieners as in definition \ref{def-realtrophurwitz}(b). If $\gamma'$ is another suitable involution, then the monodromy graph is of complementary color.
\end{lemma}
\begin{proof}
Consider a bounded edge of weight $d$ and its two adjacent vertices. If the color of the two edges to the left is equal to the color of the two edges to the right, the two corresponding transpositions have to be equal. If the color of the two edges to the left is different from the color of the two on the right, the two transpositions have to be different. It follows that we can have the same tuple $(\sigma,\tau_1,\ldots,\tau_r)$ only if the sequence of color changes coincides, which is the case only for graphs of complementary colors.
\end{proof}
\begin{lemma}\label{lem-overcount}
Let $d\equiv 0 \;\mod 4$, $\mu,\nu\in\{d, (\frac{d}{2}\frac{d}{2})\} $.
Let $\Gamma$ and $\Gamma'$ be two complementary colored chains of wieners as in definition \ref{def-realtrophurwitz}(b). Let $M_{\Gamma}$ resp.\ $M_{\Gamma'}$ denote the set of tuples $(\sigma,\tau_1,\ldots,\tau_r)$ such that there exists an involution $\gamma$ satisfying the requirements such that the monodromy graph construction \ref{const-monodromy} yields $\Gamma$ resp.\ $\Gamma'$. Then $\#M_{\Gamma}\cap M_{\Gamma'}=\frac{d!}{2}\cdot 2^W$ where $W$ denotes the number of wieners.
\end{lemma}
\begin{proof}
Denote $k=\frac{d}{2}$. Consider a permutation $\sigma$ and a transposition $\tau_1$ corresponding to the first vertex. For each coloring of the graph, we obtain a unique involution satisfying the requirements for $\sigma$ and $\tau_1$ by lemma \ref{lem-gammaunique}. If $\Gamma$ and $\Gamma'$ start with a balanced fork, there are $\frac{(2k)!}{k\cdot k\cdot 2}$ possible $\sigma$ and $k\cdot k$ possible $\tau_1$.
If $\Gamma$ and $\Gamma'$ start with an edge of weight $2k$, there are $(2k-1)!$ possible $\sigma$ and $k$ possible $\tau_1$. In any case, we have $\frac{d!}{2}$ choices for the pair $(\sigma, \tau_1)$.
Once $(\sigma,\tau_1)$ is fixed and with it, the involution $\gamma$, we can refer to our cut and join vertex multiplicities (see lemma \ref{lem-vertexmult}) to determine the number of $\tau_i$,$i=2,\ldots,r$. Since for any of the cutting vertices, the two following edges are indistinguishable, we get $1$ for the cutting vertices. We get $4$ for vertices joining two normal edges and $k$ for vertices joining two bold edges.

Now consider the following vertices one after the other. First consider the case of a cut. We can again without loss of generality assume that the permutation to cut is $(1\;\ldots\; 2k)$ and the two possible involutions corresponding to the different colorings are $(1\; k)(2\;k-1)\cdots (\frac{k}{2}\;\frac{k}{2}+1)(k+1\;2k)(k+2\;2k-1)\cdots(\frac{3k}{2}\;\frac{3k}{2}+1) $ (for two bold edges) resp.\ $(1 \;2k)(2\;2k-1)\cdots (k\;k+1)$ (for two normal edges). We can see that in each case, we need to cut with $\tau_i=(1\; k+1)$.
For any vertex joining two edges, we can again without loss of generality assume that the permutation with two cycles to join is $(1\;\ldots\; k)(k+1\;\ldots\; 2k)$ and the two possible involutions are $(1\; k)(2\;k-1)\cdots (\frac{k}{2}\;\frac{k}{2}+1)(k+1\;2k)(k+2\;2k-1)\cdots(\frac{3k}{2}\;\frac{3k}{2}+1) $ (for two bold edges) resp.\ $(1 \;2k)(2\;2k-1)\cdots (k\;k+1)$ (for two normal edges).
If the two edges are normal, we can join with $(1 \;k+1)$, $(1\; \frac{3k}{2}+1)$, $(\frac{k}{2}+1\;k+1)$ and $(\frac{k}{2}+1\;\frac{3k}{2}+1)$.
If the two edges are bold, we can join with $(1 \;k+1)$, $(2\; 2k)$, $(3\; 2k-1)$, \dots, $(k\; k+2)$.
There are two transpositions that appear in both cases, namely $(1\; k+1)$ and $(\frac{k}{2}+1\;\frac{3k}{2}+1)$.

Since there is one vertex joining two edges for every wiener in the graph, in total we obtain
$\frac{d!}{2}\cdot 2^W$ tuples that appear in the intersection.
\end{proof}

We are now ready to prove the correspondence theorem:
\begin{theorem}[Correspondence Theorem for real double Hurwitz numbers with real positive branch points (counted without real structure)]\label{thm-corres2}
 Algebro-geometric and tropical real double Hurwitz numbers with real positive branch points coincide, i.e.\ we have $$H_g(\mu,\nu)=H_g^{\trop}(\mu,\nu).$$
\end{theorem}
\begin{proof}
 The proof is exactly along the lines of the proof of the correspondence theorem \ref{thm-corres}.
It follows from lemma \ref{lem-gammaunique} and lemma \ref{lem-wienerchain} that the only graphs for which the count with real structure and without differs are chains of wieners, which can only happen if $\mu,\nu\in \{d, (\frac{d}{2},\frac{d}{2})\}$.
If $d\equiv 2 \;\mod 4$, then by lemma \ref{lem-oddwienerchain}, we can neglect the dashed wiener chain from definition \ref{def-realtrophurwitz}(a).
If $d\equiv 0 \;\mod 4$, then by lemma \ref{lem-complementary} we have to subtract from the multiplicities $\tilde{m}(\Gamma_1)+\tilde{m}(\Gamma_2)$ of a pair of chains of wieners of complementary color, to account for tuples $(\sigma,\tau_1,\ldots,\tau_r)$ that appear for both.
It follows from lemma \ref{lem-overcount} that we need to subtract $2^{W-1}$, where $W$ denotes the number of wieners.
\end{proof}

Using the tropical approach, we can easily deduce the following nice conclusion about the relation of real double Hurwitz numbers counted with or without real structures:
\begin{corollary}\label{cor-tildeh}
 If $\mu, \nu\in\{d, (\frac{d}{2},\frac{d}{2})\} $ and $d\equiv 2 \;\mod 4$, then we have
$$ H_g(\mu,\nu)= \tilde{H}_g(\mu,\nu)-\frac{1}{2}.$$
If $\mu, \nu\in\{d, (\frac{d}{2},\frac{d}{2})\} $ and $d\equiv 0 \;\mod 4$,
$$H_g(\mu,\nu)= \tilde{H}_g(\mu,\nu)-2^B 4^{g-1}.$$

Otherwise, we have
$$ H_g(\mu,\nu)= \tilde{H}_g(\mu,\nu).$$

Here, $B$ is the number of partitions of $\mu,\nu$ which are of the form $ (\frac{d}{2},\frac{d}{2})$.
\end{corollary}
\begin{proof}
 If $\mu, \nu\in\{d, (\frac{d}{2},\frac{d}{2})\} $ and $d\equiv 2 \;\mod 4$, then the only difference between the two counts is that we count the chain of wieners with dashed even edges as $0$ when counting without real structure. It is counted with multiplicity $1/2$ when counting with real structure. 
Let $\mu, \nu\in\{d, (\frac{d}{2},\frac{d}{2})\} $ and $d\equiv 0 \;\mod 4$. Then there are $2^{W+B}$ ways to color chains of wieners, so we have $2^{W+B-1}$ pairs of chains of wieners of complementary colors. For each, we have to subtract $2^{W-1}$, so altogether we subtract $2^{W+B-1}\cdot 2^{W-1}=2^B 4^{W-1}$. Obviously, for a chain of wiener, $W=g$.
It follows from the correspondence theorems \ref{thm-corres} and \ref{thm-corres2} that in all other cases, we have $H_g(\mu,\nu)=H_g^{\trop}(\mu,\nu)=\sum_{\Gamma}m(\Gamma)=\sum_{\Gamma}\tilde{m}(\Gamma)=\tilde{H}_g^{\trop}(\mu,\nu)= \tilde{H}_g(\mu,\nu)$.
\end{proof}

\begin{remark}\label{rem-multuncolor}
 Note that we can also group real monodromy graphs that differ only in the coloring in equivalence classes and define the multiplicity of an equivalence class to be the sum of the multiplicities of the elements. In this way, we can express the real double Hurwitz numbers as sums over uncolored monodromy graphs. We denote an uncolored monodromy graph by $\Gamma_{\uc}$. We set $m_{\uc}(\Gamma_{\uc})=\sum_{\Gamma} m(\Gamma)$ where $\Gamma$ goes over all colorings of $\Gamma_{\uc}$, and analogously $\tilde{m}_{\uc}(\Gamma_{\uc})=\sum_{\Gamma} \tilde{m}(\Gamma_{\uc})$.
Then obviously we have
$$ H_g^{\trop}(\mu,\nu)= \sum_{\Gamma_{\uc}} m_{\uc}(\Gamma_{\uc}), \mbox{ and }$$
$$ \tilde{H}_g^{\trop}(\mu,\nu)= \sum_{\Gamma_{\uc}} \tilde{m}_{\uc}(\Gamma_{\uc}),$$
where the sum goes over all monodromy graphs without color.

If $\Gamma_{\uc}$ is not a chain of wieners, we have (see Lemma \ref{lem-wienerchain})
\[
  m_{\uc}(\Gamma_{\uc})
    = \tilde{m}_{\uc}(\Gamma_{\uc})
    = \begin{cases}
       2^{B'} m(\Gamma) &\text{if $\Gamma_{\uc}$ has a  coloring $\Gamma$,} \\
        0 &\text{if there exists no coloring.}
      \end{cases}
\]
Here, $B'$ denotes the number of even balanced forks. Note that if there is a coloring, we can independently choose the colorings of the even forks to be bold or normal, leading to $2^{B'}$ colorings each of the same multiplicity.

If $\Gamma_{\uc}$ is a chain of wieners and $d\equiv 2 \;\mod 4$, then we have seen in Lemma \ref{lem-oddwienerchain} that there are exactly two colorings, let us call them $\Gamma$ and $\Gamma'$, where $\Gamma$ contains normal and bold edges, and $\Gamma'$ contains dashed and normal edges. It follows from Lemma \ref{lem-oddwienerchain} that
$$m_{\uc}(\Gamma_{\uc})=m(\Gamma)+m(\Gamma')=2^{-W-1}\cdot d^W+ \frac{1}{2}\mbox{ and }$$ $$\tilde{m}_{\uc}(\Gamma_{\uc})= m(\Gamma)=2^{-W-1}\cdot d^W.$$

If $\Gamma_{\uc}$ is a chain of wieners and $d\equiv 0 \;\mod 4$, then there are $N:=2^{W+B}$ colorings $\Gamma_1,\ldots,\Gamma_{N}$. Here, $W$ denotes the number of wieners and $B$ the number of balanced forks. The multiplicity of a coloring is determined by the number of bold wieners. If there are $i$ bold wieners in $\Gamma_j$, then $m(\Gamma_j)=\frac{1}{2^{W+B}}\cdot 2^{W+B-1}\cdot \big(\frac{d}{2}\big)^i\cdot 4^{W-i}$. There are $2^B \binom{W}{i}$ colorings with $i$ bold wieners.
Hence we obtain

$$m_{\uc}(\Gamma_{\uc})= 2^{B-1} \cdot \Big(\frac{d}{2}+4\Big)^W,$$
and, using Corollary \ref{cor-tildeh},
$$\tilde{m}_{\uc}(\Gamma_{\uc})= 2^{B-1}\cdot \Big(\frac{d}{2}+4\Big)^W- 2^B 4^{W-1}.$$
\end{remark}

\section{Genus zero double Hurwitz numbers}\label{sec-walls}
Inspired by the rich structure of (complex) double Hurwitz numbers \cites{gjv:ttgodhn, ssv:cbodhn, cjm:wcfdhn, Joh10}, we now study the structure of our real double Hurwitz numbers. It turns out that for real Hurwitz numbers, the structure is not as nice as in the complex world.
The methods we use here follow closely the methods developed in \cite{CJM10}, section 6.

For simplicity, we restrict our attention to genus $0$ in this section. Furthermore, we require that at least one of $\mu$ and $\nu$ has at least $3$ parts. With this restriction, we avoid the cases in which $H_g(\mu,\nu)$ and $\tilde{H}_g(\mu,\nu)$ differ by corollary \ref{cor-tildeh}.

Since the degree cancels in the Riemann-Hurwitz formula when we fix two special ramifications, it makes sense to view Hurwitz numbers as a function
\begin{equation}H_{0}:\Big\{(\mu,\nu)\in \N^{\ell(\mu)+\ell(\nu)}\;|\; \sum \mu_i=\sum \nu_i \Big\}\rightarrow \Q:(\mu,\nu)\mapsto H_{0}(\mu,\nu).\label{eq-function}\end{equation}
In fact, we could even only fix the sum of length $\ell(\mu)+\ell(\nu)=n$ and use positive and negative signs to indicate which entry belongs to which partition. This latter point of view has the advantage that we can discuss more Hurwitz numbers in a unified way. Since this aspect does not play an important role here, we choose to stick to our old notation of $\mu$ and $\nu$ and fix the lengths of both.

\begin{definition}
 We define a function $\pari:\N\rightarrow \{1,2\}$ by sending an even number to $2$ and an odd number to $1$.
\end{definition}

It follows from Lemma 6.4 of \cite{CJM10} that the weights of the edges of a monodromy graph whose weights at the ends are prescribed by $\mu$ and $\nu$ are given as signed sums of entries of $\mu$ and $\nu$:
the weight $\omega(e)$ equals $$\omega(e)=\sum_{i\in I}\mu_i - \sum_{j\in J}\nu_j$$ where $I\subset \{1,\ldots,\ell(\mu)\}$ and $J\subset \{1,\ldots,\ell(\nu)\}$ are the subsets of in- and out-ends belonging to the connected component of $\Gamma\setminus\{e\}$ from which $e$ points away.
For a fixed monodromy graph $\Gamma$, its multiplicity thus equals
$$ 2^{-B} \prod_e \pari(\omega(e))=  2^{-B}\prod_e \pari(\sum_{i\in I}\mu_i - \sum_{j\in J}\nu_j).   $$

At first glance, the map sending a tuple $(\mu,\nu)$ to the multiplicity of $\Gamma$ with the weights of the ends given by $(\mu,\nu)$ has no nice structure.
When we restrict to the set of points $(\mu,\nu)$ with all entries $\mu_i$ and $\nu_i$ even, then the weight of any interior edge is even and we can conclude $\pari(\omega(e))=2$ for all interior edges and thus the multiplicity equals $2^{\ell(\mu)+\ell(\nu)-B-3}$, since a $3$-valent graph with $ \ell(\mu)+\ell(\nu)$ ends has $\ell(\mu)+\ell(\nu)-3$ interior edges by an Euler characteristic count.

Assume that the entries of $\mu$ and $\nu$ are even and distinct, i.e.\ $|\Aut(\mu)|=|\Aut(\nu)| =1$. Consider a fixed tree $T$ with $n=\ell(\mu)+\ell(\nu)$ marked ends. We discuss the possibilities to obtain a monodromy graph from $T$ when imposing the weights $\mu$ and $\nu$ according to the labels of the ends.
First, it follows from Lemma 2.2 of \cite{BCM12} that we can find suitable edge orientations in at most one way.
We set $p(T)=1$ if there is a suitable edge orientation, and $p(T)=0$ else.
If $p(T)=1$ then we denote by $o(T)$ as in remark \ref{rem-moncount} the number of vertex orderings respecting the partial ordering imposed by the edge orientations.
Finally, as by our assumptions there are no wiener nor balanced forks in the graph, 
there is a unique admissible coloring painting all edges normal.
It follows from remark \ref{rem-moncount} (together with the correspondence theorems \ref{thm-corres} and \ref{thm-corres2} and corollary \ref{cor-tildeh}) that
\begin{equation} H_0(\mu,\nu)= 2^{n-3}\sum_{T}p(T) o(T),\label{eq-hfunction}\end{equation}
where the sum goes over all trees $T$ with $n$ marked ends. 
In the general case (i.e.\ without the assumption that $\mu$ and $\nu$ are even and have no automorphisms) the expression provides an upper bound (see also Remark \ref{rem-multuncolor}).
We can conclude the following result:
\begin{theorem}\label{thm-piecewise}
Let at least one of the partitions $\mu$ and $\nu$ have at least three parts and $n=\ell(\mu)+\ell(\nu)$.

 Genus zero real double Hurwitz numbers with positive real branch points $H_0(\mu,\nu)$ considered as a function as in equation (\ref{eq-function}) are bounded by the piecewise constant function
$$F: \N^{n}\rightarrow \Q:(\mu,\nu)\mapsto  2^{n-3}\sum_{T}p(T) o(T).$$
Here, the sum goes over all trees $T$ with $n$ marked ends and $p(T)$ and $o(T)$ are as defined above.

If all entries of $\mu$ and $\nu$ are even and $|\Aut(\mu)|=|\Aut(\nu)| =1$
(i.e.\ on $(2\N)^n\setminus \{\text{diagonals}\}$),
the function $H_0$ equals the upper bound $F$, i.e.\ we have $H_0(\mu,\nu)= F(\mu,\nu)$.

Walls of the piecewise constant function $F$ are, exactly as in the case of complex double Hurwitz numbers, given by hyperplanes of the form
$$ \{\sum_{i\in I}\mu_i - \sum_{j\in J}\nu_j=0\} .$$

\end{theorem}

\begin{proof}
 The statement about equality on $(2\N)^n\setminus \{\text{diagonals}\}$ follows from equation (\ref{eq-hfunction}) and the discussion above. When we drop the assumption that all entries are even,
$2^{n-3}$ is still an upper bound for the multiplicity of a monodromy graph (as $E \leq n-3$ and
no wiener contribute). When $|\Aut(\mu)|\cdot|\Aut(\nu)| >1$, we count in fact labelled monodromy graphs, as our trees are labelled. Again, this leads only to an overcount in $F$. The only source for possible undercounting is the existence of several colorings for the same (uncolored) monodromy graph. However, this can only occur if the graph contains balanced even forks. More precisely, the
number of possible colorings is either zero or given by $2^{B'}$, where $B'$ denotes the number
of balanced even forks. However, this is compensated by the factor $2^{-B}$ in the multiplicity
formula. So $F$ is indeed an upper bound.

Walls of the piecewise constant function appear if there is a tree for which $p(T)$ changes from $1$ to $0$, i.e.\ if there is an edge that cannot be oriented the way it used to be for the new values of $\mu$ and $\nu$. This happens exactly when there is an edge of weight $0$, i.e.\ for $(\mu,\nu)$ satisfying $ \sum_{i\in I}\mu_i - \sum_{j\in J}\nu_j=0$.
\end{proof}

Recall that for complex double Hurwitz numbers, we obtain a piecewise polynomial function for which we can express wall-crossings in terms of ``smaller'' Hurwitz numbers \cite{ssv:cbodhn}.
We should thus also try to consider wall-crossings for the upper bound piecewise constant function $F$.
It turns out however that we do not get a nice recursive structure, since the edge weight does not depend on the entries $\mu_i,\nu_i$ but equals two. This fact is responsible for a sign that prevents us from describing the vertex orderings on $\Gamma$ in terms of the vertex orderings on the components of $\Gamma\setminus e$ (where $e$ is the edge of weight $0$).

\begin{definition}
 Let $W= \{\sum_{i\in I}\mu_i - \sum_{j\in J}\nu_j=0\}$ be a wall of the piecewise constant function $F$ of theorem \ref{thm-piecewise}. Let $(\mu,\nu)$ be a point on the right of $W$ and $(\tilde{\mu},\tilde{\nu})$ a point on the left. By a wall-crossing $F_W$, we denote the value of the difference $F_W=F(\mu,\nu)-F(\tilde{\mu},\tilde{\nu})$. A wall-crossing is only defined up to sign.
\end{definition}

Consider the wall $W= \{\sum_{i\in I}\mu_i - \sum_{j\in J}\nu_j=0\}$ and let $T$ be a tree that contains an edge $e$ of weight $\sum_{i\in I}\mu_i - \sum_{j\in J}\nu_j$. We denote by $T_1$ the orientation of $T$ that appears on the right of the wall and by $T_2$ the orientation that appears on the left of the wall. These two orientations only differ in the direction of the edge $e$. We denote by $p_1(T)$ the value of $p(T)$ on the right of the wall.
\begin{theorem}
 Wall-crossing formulas for the piecewise constant function $F$ from theorem \ref{thm-piecewise} are given by
$F_W=2^{n-3}\sum_{T} p_1(T)\cdot (o(T_1)-o(T_2))$, where the sum goes over all trees $T$ with an edge of weight $\sum_{i\in I}\mu_i - \sum_{j\in J}\nu_j$.
\end{theorem}
\begin{proof}
 A tree which does not have an edge of weight $\sum_{i\in I}\mu_i - \sum_{j\in J}\nu_j$ contributes equally to both sides and thus cancels in the wall-crossing. The formula for $F_W$ is the difference of the two functions evaluated at the trees which have an edge of this weight.
\end{proof}
\begin{example}
 Let $n=5$ and consider the graph $T$ contributing to the wall-crossing $\mu_2=\nu_3$ depicted below:
\begin{center}
 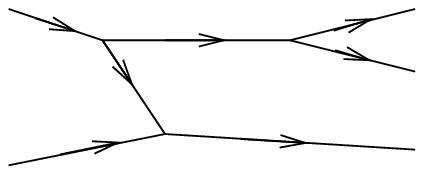 \hspace{13ex}
 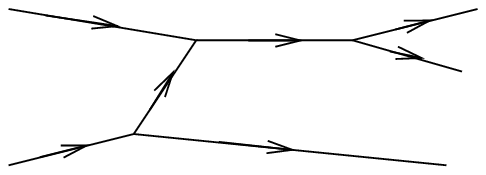
\end{center}
The two orientations of the edge with weight $\mu_2-\nu_3$ are shown on the left and right.
We have $o(T_1)=2$ since the two right vertices can appear in any order, and $o(T_2)=1$.
That is, $T$ contributes to the wall-crossing $F_W$ with $2^{2}\cdot (2-1)$.

Note that in the complex world, since the edge $e$ contributes with a weight that changes sign also when crossing the wall, we would get a contribution of $o(T_1)+o(T_2)=3$ times the product of edge weights. The sum of vertex orderings equals $\binom{3}{1}\cdot o(T\setminus e)$ --- here, $3$ is the total number of vertices and $1$ is the number of vertices appearing in the lower component of $T\setminus e$. In this way, we can break down the contribution to factors belonging to the two components of $T\setminus e$, and finally obtain an expression of the wall-crossing in terms of smaller Hurwitz numbers. The fact that in the real world, the sign of the difference does not cancel with the sign of the weight of the edge $e$ prevents us from obtaining such an expression here.
\end{example}

\section{The Cayley graph}\label{sec-cayley}

Now we will consider the tuples of lemma \ref{lem-tuples} as walks in the Cayley graph of the symmetric group generated by all transpositions, and see that the induced subgraph whose vertices are the involutions plays an important role.

\begin{definition}[Walk in the Cayley graph]\label{def-cayley-walk}
Given a tuple $(\gamma, \sigma, \tau_1, \ldots, \tau_r)$ of elements of the symmetric group $\mathbb{S}_d$ as in lemma \ref{lem-tuples}, the corresponding \emph{walk} in the (left) Cayley graph of $\mathbb{S}_d$ generated by all transpositions is
\[
  \sigma_0
  \xrightarrow{\tau_1} \sigma_1
  \xrightarrow{\tau_2} \sigma_2
  \cdots
  \xrightarrow{\tau_r} \sigma_r,
\]
where $\sigma_0 = \sigma$ and $\sigma_i = \tau_i \circ \cdots \circ \tau_1 \circ \sigma$ are the vertices of the walk. They are connected by edges labeled $\tau_1, \tau_2, \ldots, \tau_r$, since $\sigma_{i+1} = \tau_{i+1} \circ \sigma_i$ for $i = 0, \ldots, r-1$.
\end{definition}

Given a tuple $(\gamma, \sigma, \tau_1, \ldots, \tau_r)$ which satisfies the conditions (a--e) of lemma \ref{lem-tuples}, definition \ref{def-cayley-walk} lets us obtain a walk in the Cayley graph, but not all walks in the Cayley graph come from such tuples, so there is the question of how to go back. In particular, the involution $\gamma$ doesn't appear in definition \ref{def-cayley-walk}. Thankfully, most of the conditions of lemma \ref{lem-tuples} have a straightforward interpretation in terms of walks in the Cayley graph:
\begin{enumerate}
\item \textit{$\sigma$ has cycle type $\mu$;}

This holds exactly when the initial vertex $\sigma_0 = \sigma$ has cycle type $\mu$.

\item \textit{the $\tau_i$ are transpositions;}

This is built into the definition of the Cayley graph.

\item \textit{$\tau_r\circ\cdots\circ\tau_1\circ\sigma$ has cycle type $\nu$;}

This holds exactly when the final vertex $\sigma_r = \tau_r\circ\cdots\circ\tau_1\circ\sigma$ has cycle type $\nu$.

\item \textit{the subgroup generated by $\sigma,\tau_1,\ldots,\tau_r$ acts transitively on the set $\{1,\ldots,d\}$;}

This condition is the exception in having no straightforward interpretation in the Cayley graph. However, since this is essentially a condition about connected components, it is possible to ignore it at first, and then to recover it by taking the logarithm of a suitable generating function. Note also that when $\mu, \nu$ are not on a wall of the wall-crossing arrangement (that is, in the generic case), this condition is guaranteed to hold.

\item \textit{$\gamma$ is an involution (i.e.\ $\gamma^2 = \id$) satisfying $\gamma\circ\sigma\circ\gamma=\sigma^{-1}$ and $\gamma\circ (\tau_i\circ\cdots\circ\tau_1\circ\sigma)\circ\gamma = (\tau_i\circ\cdots\circ\tau_1\circ\sigma)^{-1}$ for all $i=1,\ldots,r$.}

This is the only condition which depends on $\gamma$. It can be rephrased as follows: for the `translated' walk
\[
  (\sigma_0 \circ \gamma)
  \xrightarrow{\tau_1} (\sigma_1 \circ \gamma)
  \xrightarrow{\tau_2} (\sigma_2 \circ \gamma)
  \cdots
  \xrightarrow{\tau_r} (\sigma_r \circ \gamma),
\]
obtained by precomposing each vertex of the original walk with the involution $\gamma$, all the vertices are also involutions.
This is because we have $\gamma \circ \sigma_i \circ \gamma = \sigma_i^{-1}$ if and only if $(\sigma_i \circ \gamma)^2 = \id$, that is, an involution.
Given this condition, it makes sense to consider the part of the Cayley graph whose vertices are the involutions in $\mathbb{S}_d$.
\end{enumerate}

\begin{definition}[Restricted Cayley graph]
  The \emph{restricted Cayley graph} is the induced subgraph of the Cayley graph for $\mathbb{S}_d$ (generated by all transpositions) whose vertices are exactly the involutions in $\mathbb{S}_d$.
\end{definition}

\begin{example}
The figure below illustrates the restricted Cayley graphs for $d=3$ and $d=4$. On the left, the whole Cayley graph for $d=3$ is shown, with the restricted Cayley graph highlighted. On the right, only part of the whole Cayley graph for $d=4$ is shown, but all of the restricted Cayley graph appears, again highlighted.
\[
  \includegraphics[scale=.5]{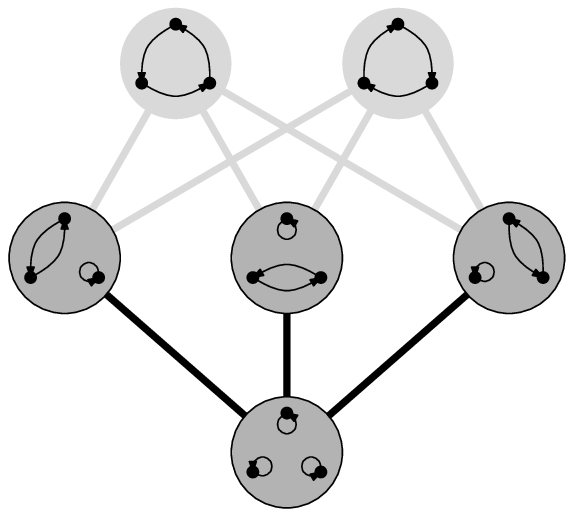}
  \qquad
  \includegraphics[scale=.5]{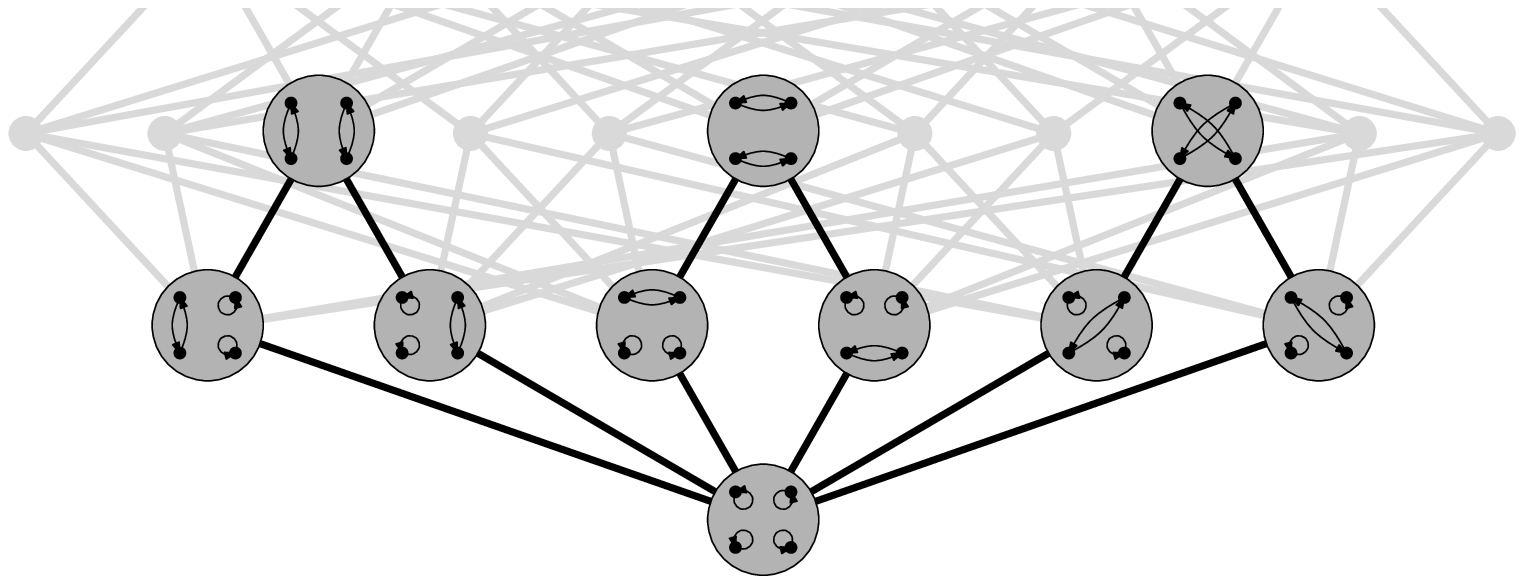}
\]
\end{example}

The relevance of this to the problem of counting the tuples of lemma \ref{lem-tuples} is summarized in the following theorem.

\begin{theorem}\label{thm:restricted-path}
  For a fixed involution $\gamma$, fixed number of transposition $r$, and fixed starting and ending permutations $\sigma_0 = \sigma$ and $\sigma_r = \sigma \circ \tau_1 \circ \cdots \circ \tau_r$, the number of tuples which satisfy the conditions of lemma \ref{lem-tuples} is equal to the number of walks
  \[
    \sigma'_0
    \xrightarrow{\tau_1} \sigma'_1
    \xrightarrow{\tau_2} \sigma'_2
    \cdots
    \xrightarrow{\tau_r} \sigma'_r
  \]
  in the restricted Cayley graph from $\sigma'_0 = \sigma_0 \circ \gamma$ to $\sigma'_r = \sigma_r \circ \gamma$ such that the subgroup generated by $\sigma, \tau_1, \ldots, \tau_r$ acts transitively on $\{1,\ldots,d\}$.
\end{theorem}

\begin{proof}
  As discussed above, the tuples which satisfy the conditions of lemma \ref{lem-tuples} can each be interpreted as a distinct walk in the (usual) Cayley graph of $\mathbb{S}_d$ of the form
  \[
    \sigma_0
    \xrightarrow{\tau_1} \sigma_1
    \xrightarrow{\tau_2} \sigma_2
    \cdots
    \xrightarrow{\tau_r} \sigma_r,
  \]
  which can be `translated' by the involution $\gamma$ to a walk
  \[
    (\sigma_0 \circ \gamma)
    \xrightarrow{\tau_1} (\sigma_1 \circ \gamma)
    \xrightarrow{\tau_2} (\sigma_2 \circ \gamma)
    \cdots
    \xrightarrow{\tau_r} (\sigma_r \circ \gamma)
  \]
  in the restricted Cayley graph. This translation by $\gamma$ can be inverted by translating by $\gamma^{-1}$, so the count is preserved. The conditions (a--c) of lemma \ref{lem-tuples} are automatically satisfied by the choice of $\sigma_0$ and $\sigma_r$; condition (d) is equivalent to the statement that the `translated' walk stays within the restricted Cayley graph; and condition (e) is repeated in the statement of theorem \ref{thm:restricted-path}.
\end{proof}

Note that the restricted Cayley graph is actually a well-known graph; if we identify each involution in $\mathbb{S}_d$ with a matching in the complete graph $K_d$ in the natural way, and consider the partial order relation given by the subgraph relation, then we are dealing with the Hasse diagram of this poset. This is because any step in the Cayley graph restricted to these vertices can only be a join of two fixed points, or a cut of a 2-cycle.

Because of this, it becomes easy to count the number paths from a vertex $\sigma'$ to a vertex $\rho'$ in the restricted Cayley graph in certain cases. For example, we deal with the case of minimum length paths (which correspond to genus 0 covers, if the connectivity requirement is satisfied) in theorem \ref{thm:walks-count-all} below.

\begin{definition}[Notation for involutions and matchings]
Let us write $I_d$ for the set of involutions in the symmetric group $\mathbb{S}_d$, $M_d$ for the set of matchings in the complete graph $K_d$, and $m(\sigma')$ for the matching which corresponds to an involution $\sigma' \in I_d$.
\end{definition}

Note that, given involutions $\sigma', \rho' \in I_d$, the union of their matchings $m(\sigma')$ and $m(\rho')$ has a special structure: its connected components consist of alternating paths and alternating cycles, that is, paths and cycles where every edge from $m(\sigma')$ is followed by an edge from $m(\rho')$ and vice versa. This is not quite rigorous in the case of a component
which consists of a single edge contained in both matchings. By slight abuse of terminology, 
we will consider such a component as an alternating cycle of length 2.

\begin{lemma}\label{lem:walks-count-connected}
  Let $\sigma', \rho' \in I_d$ be involutions, and suppose that $m(\sigma') \cup m(\rho')$ consists of a single connected component. Let $p_0$ be the number of minimum length paths from $\sigma'$ to $\rho'$ in the restricted Cayley graph.

  If $m(\sigma') \cup m(\rho')$ is an alternating path on $d$ vertices, then $p_0$ is the coefficient of $x^{d-1} / (d-1)!$ in the generating function
  \[
    \sec(x) + \tan(x)
      = 1 + x + \frac{x^2}{2} + 2\frac{x^3}{6} + 5\frac{x^4}{24} + 16\frac{x^5}{120} + 61\frac{x^6}{720} + \cdots
  \]

  If $m(\sigma') \cup m(\rho')$ is an alternating cycle on $d$ vertices, then $p_0$ is the coefficient of $x^d / d!$ in the generating function
  \[
    \frac{x \tan(x)}{2} = \frac{x^2}{2} + 4\frac{x^4}{24} + 48\frac{x^6}{720} + 1088\frac{x^8}{40320} + 39680\frac{x^{10}}{3628800} + \cdots
  \]
\end{lemma}

\begin{proof}
  As noted above, a path from $\sigma'$ to $\rho'$ in the restricted Cayley graph corresponds to a sequence a of matchings in $M_d$, starting from $m(\sigma')$ and ending with $m(\rho')$, where at each step we obtain the next matching by either cutting an edge from the current matching or joining two currently unmatched vertices.

  In particular, such a path must cut every edge which is in $m(\sigma') \setminus m(\rho')$, and create every edge which is in $m(\rho') \setminus m(\sigma')$, so a path will have minimum length if it does exactly this and nothing more. So, to count minimum length paths, it suffices to count the valid ways of putting these cuts and joins in sequence.
  The only constraint on the ordering of cuts and joins is that before joining two vertices, these vertices must be unmatched, so any edges involving them must be cut.

  Now, suppose $m(\sigma') \cup m(\rho')$ is an alternating path on $d$ vertices.
  For $d=1$, that is, a path with one vertex and no edges, there is trivially exactly one minimum length path. For $d \geq 2$, by exchanging the roles of $\sigma'$ and $\rho'$ if needed, we may assume that the alternating path starts with an edge of $m(\sigma')$, then (possibly) an edge of $m(\rho')$, and so on. Let us number these edges $1, 2, \ldots, d-1$, and write $t_1, t_2, \ldots, t_{d-1}$ for time at which each edge is either cut or created. Then, the ordering constraint is
  \[
    t_1 < t_2 > t_3 < t_4 > \cdots,
  \]
  that is, the sequence $t_1, t_2, \ldots, t_{d-1}$ must be an alternating permutation. The exponential generating function for these is (see, e.g., \cite{gj:ce}*{Section 3.2.22})
  \[
    \sec(x) + \tan(x).
  \]

  If $m(\sigma') \cup m(\rho')$ is an alternating cycle, then it must have even length, and there are exactly $d/2$ choices for for which edge is created last. If we number the edges $1, 2, \ldots, d$ around the cycle so that this last created edge is $d$, then we again have the ordering constraint
  \[
    t_1 < t_2 > t_3 < \cdots > t_{d-1}
  \]
  for the other $d-1$ edges. The exponential generating function for alternating permutations of odd length is simply
  \[
    \tan(x).
  \]
  To account for the shift from $d-1$ to $d$ and the extra factor of $d/2$, it suffices to multiply this generating function by $x/2$, yielding the generating function
  \[
    \frac{x \tan(x)}{2}
  \]
  from the claim.
\end{proof}

\begin{theorem}\label{thm:walks-count-all}
  Let $\sigma', \rho' \in I_d$ be involutions, and suppose that $m(\sigma') \cup m(\rho')$ consists of a alternating paths with $a_1, a_2, \ldots, a_m$ vertices respectively and alternating cycles with $b_1, b_2, \ldots, b_n$ vertices respectively. Let
  \[
    P(x) = \sec(x) + \tan(x), \qquad
    C(x) = \frac{x \tan(x)}{2}
  \]
  be the generating functions from lemma \ref{lem:walks-count-connected}.
  Then, the number of minimum length paths from $\sigma'$ to $\rho'$ in the restricted Cayley graph is the coefficient of
  \[
    x_1^{a_1-1} x_2^{a_2-1} \cdots x_m^{a_m-1} y_1^{b_1} y_2^{b_2} \cdots y_n^{b_n} / (d-m)!
  \]
  in the generating function
  \[
    P(x_1) P(x_2) \cdots P(x_m) C(y_1) C(y_2) \cdots C(y_n).
  \]
\end{theorem}

\begin{proof}
  As in the proof of lemma \ref{lem:walks-count-connected}, what matters is the number of ways of ordering the steps in which the edges of $m(\sigma') \setminus m(\rho')$ are cut and the edges of $m(\rho') \setminus m(\sigma')$ are created. In total, there are $d-m$ steps to be taken, of which $a_1-1, a_2-1, \ldots, a_m-1$ respectively are for the alternating paths, and $b_1, b_2, \ldots, b_n$ respectively are for the alternating cycles. By lemma \ref{lem:walks-count-connected}, the coefficients of the generating functions $P(x_i)$ and $C(y_j)$ count the ways of ordering the steps within each connected component. Since there are no constraints at all for ordering the steps between connected components, the number of ways of interleaving a given list of orderings for all the connected components is simply the multinomial coefficient
  \[
  \frac{(d-m)!}{(a_1-1)! (a_2-1)! \cdots (a_m-1)! b_1! b_2! \cdots b_n!},
  \]
  as needed.
\end{proof}

We are also interested in grouping together tuples which correspond to the same monodromy graph, as constructed in lemma \ref{sec:mon-graph}, so let us now discuss what happens on that front when we deal with translated walks.

\begin{theorem}
Given a (non-translated) Cayley walk
\[
  \sigma_0
  \xrightarrow{\tau_1} \sigma_1
  \xrightarrow{\tau_2} \sigma_2
  \cdots
  \xrightarrow{\tau_r} \sigma_r
\]
and the corresponding monodromy graph $\Gamma$ from construction \ref{const-monodromy}, each vertex $\sigma_i$ of the walk corresponds to a vertical cross-section of $\Gamma$, and furthermore the cycles of $\sigma_i$ correspond to the edges of $\Gamma$ which intersect the vertical cross-section.
The edges of $\Gamma$ can be of different types, namely:
\begin{enumerate}
  \item odd normal edges,
  \item even normal edges,
  \item even dashed edges, or
  \item paired up bold edges.
\end{enumerate}
The cycles of $\sigma_i$ are given by the connected components of the union of matchings $m(\sigma'_i) \cup m(\gamma)$, where $\sigma'_i = \sigma_i \circ \gamma$, which can be of different types, namely:
\begin{enumerate}
  \item alternating paths with as many edges from $m(\sigma'_i)$ as $m(\gamma)$,
  \item alternating paths with one less edge from $m(\sigma'_i)$ than $m(\gamma)$,
  \item alternating paths with one more edge from $m(\sigma'_i)$ than $m(\gamma)$, or
  \item alternating cycles with as many edges from $m(\sigma'_i)$ as $m(\gamma)$,
\end{enumerate}
as illustrated in the figure below (where blue edges are from $m(\sigma'_i)$ and red edges are from $m(\gamma)$). The types for edges of $\Gamma$ correspond exactly to the types for connected components of $m(\sigma'_i) \cup m(\gamma)$.
\[
  \includegraphics{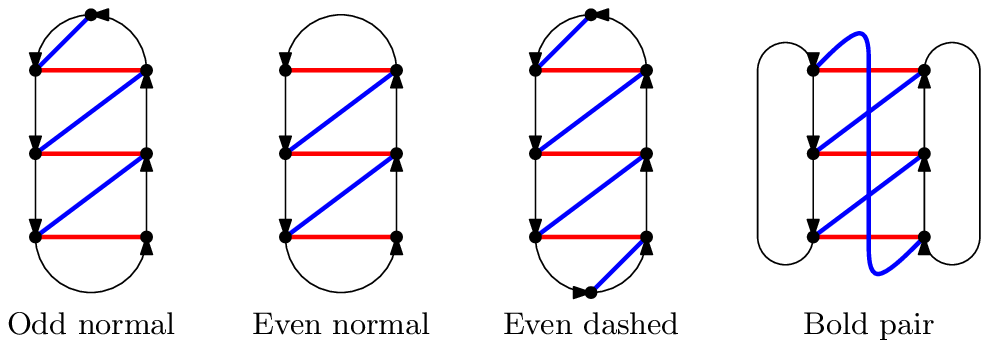}
\]
\end{theorem}

\begin{proof}
  We have already shown in construction \ref{const-monodromy} that the stated correspondence between edges of $\Gamma$ and cycles of $\sigma_i$ holds, so the content of this theorem is the relation between types of edges of $\Gamma$ and types of connected components of $m(\sigma'_i) \cup m(\gamma)$.

  Note that the $\gamma$-fixed points are the isolated vertices of $m(\gamma)$, so they are vertices of degree less than 2 in the union $m(\sigma'_i) \cup m(\gamma)$. In particular, they are endpoints of alternating paths in $m(\sigma'_i) \cup m(\gamma)$.

  The type of an edge of $\Gamma$ is related to the action of the involution $\gamma$ on the corresponding cycle of $\sigma_i$, and we can read off the action of $\gamma$ on the cycles of $\sigma_i$ by looking at the connected components of $m(\sigma'_i) \cup m(\gamma)$, as follows:
  \begin{enumerate}
    \item
      If the component is an alternating path with as many edges from $m(\sigma'_i)$ as $m(\gamma)$, then it has an odd number of vertices, and it has one endpoint which is not incident to an edge of $m(\gamma)$, which is a $\gamma$-fixed point. Also, its vertices form a single cycle of $\sigma_i$, so it corresponds to an odd normal edge of $\Gamma$.

    \item
      If the component is an alternating path with one less edge from $m(\sigma'_i)$ than $m(\gamma)$, then it has an even number of vertices, and both of its endpoints are incident to edges of $m(\gamma)$, so it has no $\gamma$-fixed points. Also, its vertices form a single cycle of $\sigma_i$, so it corresponds to an even normal edge of $\Gamma$.

    \item
      If the component is an alternating path with one more edge from $m(\sigma'_i)$ than $m(\gamma)$, then it has an even number of vertices, and neither of its endpoints is incident to an edge of $m(\gamma)$, so it has two $\gamma$-fixed points. Also, its vertices form a single cycle of $\sigma_i$, so it corresponds to an even dashed edge of $\Gamma$.

    \item
      If the component is an alternating cycle, then it has an even number of vertices, which form two interleaved cycles of $\sigma_i$ of the same length. These cycles are switched by the action of $\gamma$, so they correspond to a pair of bold edges of $\Gamma$.
  \end{enumerate}

  This covers all the possible connected components of $m(\sigma'_i) \cup m(\gamma)$.
\end{proof}

\begin{remark}
From this, it is clear that each step
\[
  \sigma_{i-1} \xrightarrow{\tau_i} \sigma_i
\]
in the non-translated Cayley walk is a cut (resp.~a join) if and only if the step
\[
  \sigma'_{i-1} \xrightarrow{\tau_i} \sigma'_i
\]
in the translated walk is a cut (resp.~a join), unless a pair of bold edges is involved, in which case it is a join (resp.~a cut).
Thus, the matchings approach gives another simple explanation of the vertex types appearing in lemma \ref{lem-vertexmult}. At each vertex in the monodromy graph, the matching $m(\sigma'_{i-1})$ changes into the matching $m(\sigma'_i)$ by adding or removing an edge. The effect of this change on the alternating paths and alternating cycles (a--d) is exactly captured by the 8 vertex types of lemma \ref{lem-vertexmult}. For example, starting with an alternating path of type (b) and removing one $m(\sigma'_{i-1})$-edge, we get two paths of type (a).
\end{remark}

\bibliographystyle{plain}
\bibliography{bibliographie}
\end {document}

%% file: figs/arrows.pstex_t
\begin{picture}(0,0)%
\includegraphics{arrows.pstex}%
\end{picture}%
\setlength{\unitlength}{3947sp}%
\begingroup\makeatletter\ifx\SetFigFont\undefined%
\gdef\SetFigFont#1#2#3#4#5{%
  \reset@font\fontsize{#1}{#2pt}%
  \fontfamily{#3}\fontseries{#4}\fontshape{#5}%
  \selectfont}%
\fi\endgroup%
\begin{picture}(4599,1881)(5089,-7255)
\put(9076,-7186){\makebox(0,0)[lb]{\smash{{\SetFigFont{12}{14.4}{\familydefault}{\mddefault}{\updefault}{\color[rgb]{0,0,0}$L$}%
}}}}
\put(6076,-7186){\makebox(0,0)[lb]{\smash{{\SetFigFont{12}{14.4}{\familydefault}{\mddefault}{\updefault}{\color[rgb]{0,0,0}$L$}%
}}}}
\put(5851,-6511){\makebox(0,0)[lb]{\smash{{\SetFigFont{12}{14.4}{\familydefault}{\mddefault}{\updefault}{\color[rgb]{0,0,0}$\pi$}%
}}}}
\put(8851,-6511){\makebox(0,0)[lb]{\smash{{\SetFigFont{12}{14.4}{\familydefault}{\mddefault}{\updefault}{\color[rgb]{0,0,0}$\pi$}%
}}}}
\end{picture}%

%% file: figs/vertextypes.pstex_t
\begin{picture}(0,0)%
\includegraphics{figs/vertextypes.pstex}%
\end{picture}%
\setlength{\unitlength}{3947sp}%
\begingroup\makeatletter\ifx\SetFigFont\undefined%
\gdef\SetFigFont#1#2#3#4#5{%
  \reset@font\fontsize{#1}{#2pt}%
  \fontfamily{#3}\fontseries{#4}\fontshape{#5}%
  \selectfont}%
\fi\endgroup%
\begin{picture}(3102,1860)(4636,-7201)
\put(7201,-7186){\makebox(0,0)[lb]{\smash{{\SetFigFont{9}{10.8}{\familydefault}{\mddefault}{\updefault}{\color[rgb]{0,0,0}even}%
}}}}
\put(7351,-5461){\makebox(0,0)[lb]{\smash{{\SetFigFont{9}{10.8}{\familydefault}{\mddefault}{\updefault}{\color[rgb]{0,0,0}odd}%
}}}}
\put(7276,-6136){\makebox(0,0)[lb]{\smash{{\SetFigFont{9}{10.8}{\familydefault}{\mddefault}{\updefault}{\color[rgb]{0,0,0}odd}%
}}}}
\put(6301,-5686){\makebox(0,0)[lb]{\smash{{\SetFigFont{9}{10.8}{\familydefault}{\mddefault}{\updefault}{\color[rgb]{0,0,0}even}%
}}}}
\put(4651,-6736){\makebox(0,0)[lb]{\smash{{\SetFigFont{9}{10.8}{\familydefault}{\mddefault}{\updefault}{\color[rgb]{0,0,0}odd}%
}}}}
\put(5626,-6511){\makebox(0,0)[lb]{\smash{{\SetFigFont{9}{10.8}{\familydefault}{\mddefault}{\updefault}{\color[rgb]{0,0,0}odd}%
}}}}
\put(4651,-5686){\makebox(0,0)[lb]{\smash{{\SetFigFont{9}{10.8}{\familydefault}{\mddefault}{\updefault}{\color[rgb]{0,0,0}even}%
}}}}
\put(5551,-7186){\makebox(0,0)[lb]{\smash{{\SetFigFont{9}{10.8}{\familydefault}{\mddefault}{\updefault}{\color[rgb]{0,0,0}even}%
}}}}
\put(6376,-6736){\makebox(0,0)[lb]{\smash{{\SetFigFont{9}{10.8}{\familydefault}{\mddefault}{\updefault}{\color[rgb]{0,0,0}even}%
}}}}
\put(7276,-6511){\makebox(0,0)[lb]{\smash{{\SetFigFont{9}{10.8}{\familydefault}{\mddefault}{\updefault}{\color[rgb]{0,0,0}even}%
}}}}
\end{picture}%

%% file: figs/wienerchain1.pstex_t
\begin{picture}(0,0)%
\includegraphics{wienerchain1.pstex}%
\end{picture}%
\setlength{\unitlength}{3947sp}%
\begingroup\makeatletter\ifx\SetFigFont\undefined%
\gdef\SetFigFont#1#2#3#4#5{%
  \reset@font\fontsize{#1}{#2pt}%
  \fontfamily{#3}\fontseries{#4}\fontshape{#5}%
  \selectfont}%
\fi\endgroup%
\begin{picture}(4149,174)(3739,-5398)
\end{picture}%

%% file: figs/wienerchain2.pstex_t
\begin{picture}(0,0)%
\includegraphics{wienerchain2.pstex}%
\end{picture}%
\setlength{\unitlength}{3947sp}%
\begingroup\makeatletter\ifx\SetFigFont\undefined%
\gdef\SetFigFont#1#2#3#4#5{%
  \reset@font\fontsize{#1}{#2pt}%
  \fontfamily{#3}\fontseries{#4}\fontshape{#5}%
  \selectfont}%
\fi\endgroup%
\begin{picture}(4149,216)(3739,-5419)
\end{picture}%

%% file: figs/h1422.pstex_t
\begin{picture}(0,0)%
\includegraphics{h1422.pstex}%
\end{picture}%
\setlength{\unitlength}{3947sp}%
\begingroup\makeatletter\ifx\SetFigFont\undefined%
\gdef\SetFigFont#1#2#3#4#5{%
  \reset@font\fontsize{#1}{#2pt}%
  \fontfamily{#3}\fontseries{#4}\fontshape{#5}%
  \selectfont}%
\fi\endgroup%
\begin{picture}(5730,2875)(3661,-7700)
\put(6076,-5086){\makebox(0,0)[lb]{\smash{{\SetFigFont{9}{10.8}{\familydefault}{\mddefault}{\updefault}{\color[rgb]{0,0,0}$2$}%
}}}}
\put(6076,-5386){\makebox(0,0)[lb]{\smash{{\SetFigFont{9}{10.8}{\familydefault}{\mddefault}{\updefault}{\color[rgb]{0,0,0}$2$}%
}}}}
\put(9376,-5086){\makebox(0,0)[lb]{\smash{{\SetFigFont{9}{10.8}{\familydefault}{\mddefault}{\updefault}{\color[rgb]{0,0,0}$2$}%
}}}}
\put(9376,-5386){\makebox(0,0)[lb]{\smash{{\SetFigFont{9}{10.8}{\familydefault}{\mddefault}{\updefault}{\color[rgb]{0,0,0}$2$}%
}}}}
\put(6076,-5986){\makebox(0,0)[lb]{\smash{{\SetFigFont{9}{10.8}{\familydefault}{\mddefault}{\updefault}{\color[rgb]{0,0,0}$2$}%
}}}}
\put(6076,-6286){\makebox(0,0)[lb]{\smash{{\SetFigFont{9}{10.8}{\familydefault}{\mddefault}{\updefault}{\color[rgb]{0,0,0}$2$}%
}}}}
\put(9376,-5986){\makebox(0,0)[lb]{\smash{{\SetFigFont{9}{10.8}{\familydefault}{\mddefault}{\updefault}{\color[rgb]{0,0,0}$2$}%
}}}}
\put(9376,-6286){\makebox(0,0)[lb]{\smash{{\SetFigFont{9}{10.8}{\familydefault}{\mddefault}{\updefault}{\color[rgb]{0,0,0}$2$}%
}}}}
\put(6076,-6811){\makebox(0,0)[lb]{\smash{{\SetFigFont{9}{10.8}{\familydefault}{\mddefault}{\updefault}{\color[rgb]{0,0,0}$2$}%
}}}}
\put(6076,-7411){\makebox(0,0)[lb]{\smash{{\SetFigFont{9}{10.8}{\familydefault}{\mddefault}{\updefault}{\color[rgb]{0,0,0}$2$}%
}}}}
\put(9376,-7186){\makebox(0,0)[lb]{\smash{{\SetFigFont{9}{10.8}{\familydefault}{\mddefault}{\updefault}{\color[rgb]{0,0,0}$2$}%
}}}}
\put(9376,-6736){\makebox(0,0)[lb]{\smash{{\SetFigFont{9}{10.8}{\familydefault}{\mddefault}{\updefault}{\color[rgb]{0,0,0}$2$}%
}}}}
\put(8551,-7636){\makebox(0,0)[b]{\smash{{\SetFigFont{12}{14.4}{\familydefault}{\mddefault}{\updefault}{\color[rgb]{0,0,0}$1$}%
}}}}
\put(5251,-5611){\makebox(0,0)[b]{\smash{{\SetFigFont{12}{14.4}{\familydefault}{\mddefault}{\updefault}{\color[rgb]{0,0,0}$2$}%
}}}}
\put(8551,-5536){\makebox(0,0)[b]{\smash{{\SetFigFont{12}{14.4}{\familydefault}{\mddefault}{\updefault}{\color[rgb]{0,0,0}$1$}%
}}}}
\put(5251,-6511){\makebox(0,0)[b]{\smash{{\SetFigFont{12}{14.4}{\familydefault}{\mddefault}{\updefault}{\color[rgb]{0,0,0}$1$}%
}}}}
\put(8551,-6511){\makebox(0,0)[b]{\smash{{\SetFigFont{12}{14.4}{\familydefault}{\mddefault}{\updefault}{\color[rgb]{0,0,0}$2$}%
}}}}
\put(5251,-7636){\makebox(0,0)[b]{\smash{{\SetFigFont{12}{14.4}{\familydefault}{\mddefault}{\updefault}{\color[rgb]{0,0,0}$1$}%
}}}}
\put(8626,-6961){\makebox(0,0)[b]{\smash{{\SetFigFont{9}{10.8}{\familydefault}{\mddefault}{\updefault}{\color[rgb]{0,0,0}$1$}%
}}}}
\put(8626,-7336){\makebox(0,0)[b]{\smash{{\SetFigFont{9}{10.8}{\familydefault}{\mddefault}{\updefault}{\color[rgb]{0,0,0}$1$}%
}}}}
\put(3676,-5236){\makebox(0,0)[rb]{\smash{{\SetFigFont{9}{10.8}{\familydefault}{\mddefault}{\updefault}{\color[rgb]{0,0,0}$4$}%
}}}}
\put(3676,-6136){\makebox(0,0)[rb]{\smash{{\SetFigFont{9}{10.8}{\familydefault}{\mddefault}{\updefault}{\color[rgb]{0,0,0}$4$}%
}}}}
\put(3676,-7036){\makebox(0,0)[rb]{\smash{{\SetFigFont{9}{10.8}{\familydefault}{\mddefault}{\updefault}{\color[rgb]{0,0,0}$4$}%
}}}}
\put(6976,-5236){\makebox(0,0)[rb]{\smash{{\SetFigFont{9}{10.8}{\familydefault}{\mddefault}{\updefault}{\color[rgb]{0,0,0}$4$}%
}}}}
\put(6976,-6136){\makebox(0,0)[rb]{\smash{{\SetFigFont{9}{10.8}{\familydefault}{\mddefault}{\updefault}{\color[rgb]{0,0,0}$4$}%
}}}}
\put(6976,-7036){\makebox(0,0)[rb]{\smash{{\SetFigFont{9}{10.8}{\familydefault}{\mddefault}{\updefault}{\color[rgb]{0,0,0}$4$}%
}}}}
\put(4651,-4936){\makebox(0,0)[b]{\smash{{\SetFigFont{9}{10.8}{\familydefault}{\mddefault}{\updefault}{\color[rgb]{0,0,0}$2$}%
}}}}
\put(4651,-5461){\makebox(0,0)[b]{\smash{{\SetFigFont{9}{10.8}{\familydefault}{\mddefault}{\updefault}{\color[rgb]{0,0,0}$2$}%
}}}}
\put(5251,-5086){\makebox(0,0)[b]{\smash{{\SetFigFont{9}{10.8}{\familydefault}{\mddefault}{\updefault}{\color[rgb]{0,0,0}$4$}%
}}}}
\put(4651,-5836){\makebox(0,0)[b]{\smash{{\SetFigFont{9}{10.8}{\familydefault}{\mddefault}{\updefault}{\color[rgb]{0,0,0}$2$}%
}}}}
\put(4651,-6361){\makebox(0,0)[b]{\smash{{\SetFigFont{9}{10.8}{\familydefault}{\mddefault}{\updefault}{\color[rgb]{0,0,0}$2$}%
}}}}
\put(5251,-5986){\makebox(0,0)[b]{\smash{{\SetFigFont{9}{10.8}{\familydefault}{\mddefault}{\updefault}{\color[rgb]{0,0,0}$4$}%
}}}}
\put(4801,-6811){\makebox(0,0)[b]{\smash{{\SetFigFont{9}{10.8}{\familydefault}{\mddefault}{\updefault}{\color[rgb]{0,0,0}$1$}%
}}}}
\put(5176,-7111){\makebox(0,0)[b]{\smash{{\SetFigFont{9}{10.8}{\familydefault}{\mddefault}{\updefault}{\color[rgb]{0,0,0}$1$}%
}}}}
\put(4501,-7261){\makebox(0,0)[b]{\smash{{\SetFigFont{9}{10.8}{\familydefault}{\mddefault}{\updefault}{\color[rgb]{0,0,0}$3$}%
}}}}
\put(7951,-4936){\makebox(0,0)[b]{\smash{{\SetFigFont{9}{10.8}{\familydefault}{\mddefault}{\updefault}{\color[rgb]{0,0,0}$2$}%
}}}}
\put(7951,-5461){\makebox(0,0)[b]{\smash{{\SetFigFont{9}{10.8}{\familydefault}{\mddefault}{\updefault}{\color[rgb]{0,0,0}$2$}%
}}}}
\put(8551,-5086){\makebox(0,0)[b]{\smash{{\SetFigFont{9}{10.8}{\familydefault}{\mddefault}{\updefault}{\color[rgb]{0,0,0}$4$}%
}}}}
\put(8551,-5986){\makebox(0,0)[b]{\smash{{\SetFigFont{9}{10.8}{\familydefault}{\mddefault}{\updefault}{\color[rgb]{0,0,0}$4$}%
}}}}
\put(7951,-5836){\makebox(0,0)[b]{\smash{{\SetFigFont{9}{10.8}{\familydefault}{\mddefault}{\updefault}{\color[rgb]{0,0,0}$2$}%
}}}}
\put(7951,-6361){\makebox(0,0)[b]{\smash{{\SetFigFont{9}{10.8}{\familydefault}{\mddefault}{\updefault}{\color[rgb]{0,0,0}$2$}%
}}}}
\put(8176,-6811){\makebox(0,0)[b]{\smash{{\SetFigFont{9}{10.8}{\familydefault}{\mddefault}{\updefault}{\color[rgb]{0,0,0}$2$}%
}}}}
\put(8026,-7186){\makebox(0,0)[b]{\smash{{\SetFigFont{9}{10.8}{\familydefault}{\mddefault}{\updefault}{\color[rgb]{0,0,0}$2$}%
}}}}
\end{picture}%

%% file: figs/arrowsgamma.pstex_t
\begin{picture}(0,0)%
\includegraphics{arrowsgamma.pstex}%
\end{picture}%
\setlength{\unitlength}{4144sp}%
\begingroup\makeatletter\ifx\SetFigFont\undefined%
\gdef\SetFigFont#1#2#3#4#5{%
  \reset@font\fontsize{#1}{#2pt}%
  \fontfamily{#3}\fontseries{#4}\fontshape{#5}%
  \selectfont}%
\fi\endgroup%
\begin{picture}(3394,1246)(5206,-3635)
\put(5716,-2851){\makebox(0,0)[lb]{\smash{{\SetFigFont{12}{14.4}{\familydefault}{\mddefault}{\updefault}{\color[rgb]{0,0,0}3)}%
}}}}
\put(6166,-2851){\makebox(0,0)[lb]{\smash{{\SetFigFont{12}{14.4}{\familydefault}{\mddefault}{\updefault}{\color[rgb]{0,0,0}(4}%
}}}}
\put(6436,-2851){\makebox(0,0)[lb]{\smash{{\SetFigFont{12}{14.4}{\familydefault}{\mddefault}{\updefault}{\color[rgb]{0,0,0}5}%
}}}}
\put(6616,-2851){\makebox(0,0)[lb]{\smash{{\SetFigFont{12}{14.4}{\familydefault}{\mddefault}{\updefault}{\color[rgb]{0,0,0}6)}%
}}}}
\put(5491,-2851){\makebox(0,0)[lb]{\smash{{\SetFigFont{12}{14.4}{\familydefault}{\mddefault}{\updefault}{\color[rgb]{0,0,0}2}%
}}}}
\put(5221,-2851){\makebox(0,0)[lb]{\smash{{\SetFigFont{12}{14.4}{\familydefault}{\mddefault}{\updefault}{\color[rgb]{0,0,0}(1}%
}}}}
\put(7336,-2851){\makebox(0,0)[lb]{\smash{{\SetFigFont{12}{14.4}{\familydefault}{\mddefault}{\updefault}{\color[rgb]{0,0,0}(1}%
}}}}
\put(7606,-2851){\makebox(0,0)[lb]{\smash{{\SetFigFont{12}{14.4}{\familydefault}{\mddefault}{\updefault}{\color[rgb]{0,0,0}2}%
}}}}
\put(7831,-2851){\makebox(0,0)[lb]{\smash{{\SetFigFont{12}{14.4}{\familydefault}{\mddefault}{\updefault}{\color[rgb]{0,0,0}3}%
}}}}
\put(8056,-2851){\makebox(0,0)[lb]{\smash{{\SetFigFont{12}{14.4}{\familydefault}{\mddefault}{\updefault}{\color[rgb]{0,0,0}4}%
}}}}
\put(8281,-2851){\makebox(0,0)[lb]{\smash{{\SetFigFont{12}{14.4}{\familydefault}{\mddefault}{\updefault}{\color[rgb]{0,0,0}5)}%
}}}}
\put(5221,-3571){\makebox(0,0)[lb]{\smash{{\SetFigFont{12}{14.4}{\familydefault}{\mddefault}{\updefault}{\color[rgb]{0,0,0}(1}%
}}}}
\put(5491,-3571){\makebox(0,0)[lb]{\smash{{\SetFigFont{12}{14.4}{\familydefault}{\mddefault}{\updefault}{\color[rgb]{0,0,0}2}%
}}}}
\put(5716,-3571){\makebox(0,0)[lb]{\smash{{\SetFigFont{12}{14.4}{\familydefault}{\mddefault}{\updefault}{\color[rgb]{0,0,0}3}%
}}}}
\put(5986,-3571){\makebox(0,0)[lb]{\smash{{\SetFigFont{12}{14.4}{\familydefault}{\mddefault}{\updefault}{\color[rgb]{0,0,0}4}%
}}}}
\put(6211,-3571){\makebox(0,0)[lb]{\smash{{\SetFigFont{12}{14.4}{\familydefault}{\mddefault}{\updefault}{\color[rgb]{0,0,0}5}%
}}}}
\put(6436,-3571){\makebox(0,0)[lb]{\smash{{\SetFigFont{12}{14.4}{\familydefault}{\mddefault}{\updefault}{\color[rgb]{0,0,0}6)}%
}}}}
\put(7336,-3571){\makebox(0,0)[lb]{\smash{{\SetFigFont{12}{14.4}{\familydefault}{\mddefault}{\updefault}{\color[rgb]{0,0,0}(1}%
}}}}
\put(7786,-3571){\makebox(0,0)[lb]{\smash{{\SetFigFont{12}{14.4}{\familydefault}{\mddefault}{\updefault}{\color[rgb]{0,0,0}3}%
}}}}
\put(8011,-3571){\makebox(0,0)[lb]{\smash{{\SetFigFont{12}{14.4}{\familydefault}{\mddefault}{\updefault}{\color[rgb]{0,0,0}4}%
}}}}
\put(8236,-3571){\makebox(0,0)[lb]{\smash{{\SetFigFont{12}{14.4}{\familydefault}{\mddefault}{\updefault}{\color[rgb]{0,0,0}5}%
}}}}
\put(8416,-3571){\makebox(0,0)[lb]{\smash{{\SetFigFont{12}{14.4}{\familydefault}{\mddefault}{\updefault}{\color[rgb]{0,0,0}6)}%
}}}}
\put(7606,-3571){\makebox(0,0)[lb]{\smash{{\SetFigFont{12}{14.4}{\familydefault}{\mddefault}{\updefault}{\color[rgb]{0,0,0}2}%
}}}}
\end{picture}%

%% file: figs/vertexmult.pstex_t
\begin{picture}(0,0)%
\includegraphics{vertexmult.pstex}%
\end{picture}%
\setlength{\unitlength}{3947sp}%
\begingroup\makeatletter\ifx\SetFigFont\undefined%
\gdef\SetFigFont#1#2#3#4#5{%
  \reset@font\fontsize{#1}{#2pt}%
  \fontfamily{#3}\fontseries{#4}\fontshape{#5}%
  \selectfont}%
\fi\endgroup%
\begin{picture}(5130,1615)(6061,-7412)
\put(11026,-6286){\makebox(0,0)[lb]{\smash{{\SetFigFont{8}{9.6}{\familydefault}{\mddefault}{\updefault}{\color[rgb]{0,0,0}$2k$}%
}}}}
\put(6076,-6884){\makebox(0,0)[lb]{\smash{{\SetFigFont{7}{8.4}{\familydefault}{\mddefault}{\updefault}{\color[rgb]{0,0,0}odd}%
}}}}
\put(6838,-6708){\makebox(0,0)[lb]{\smash{{\SetFigFont{7}{8.4}{\familydefault}{\mddefault}{\updefault}{\color[rgb]{0,0,0}odd}%
}}}}
\put(6076,-6063){\makebox(0,0)[lb]{\smash{{\SetFigFont{7}{8.4}{\familydefault}{\mddefault}{\updefault}{\color[rgb]{0,0,0}even}%
}}}}
\put(6779,-7235){\makebox(0,0)[lb]{\smash{{\SetFigFont{7}{8.4}{\familydefault}{\mddefault}{\updefault}{\color[rgb]{0,0,0}even}%
}}}}
\put(7424,-6884){\makebox(0,0)[lb]{\smash{{\SetFigFont{7}{8.4}{\familydefault}{\mddefault}{\updefault}{\color[rgb]{0,0,0}even}%
}}}}
\put(8128,-6708){\makebox(0,0)[lb]{\smash{{\SetFigFont{7}{8.4}{\familydefault}{\mddefault}{\updefault}{\color[rgb]{0,0,0}even}%
}}}}
\put(8069,-7235){\makebox(0,0)[lb]{\smash{{\SetFigFont{7}{8.4}{\familydefault}{\mddefault}{\updefault}{\color[rgb]{0,0,0}even}%
}}}}
\put(8186,-5887){\makebox(0,0)[lb]{\smash{{\SetFigFont{7}{8.4}{\familydefault}{\mddefault}{\updefault}{\color[rgb]{0,0,0}odd}%
}}}}
\put(8128,-6415){\makebox(0,0)[lb]{\smash{{\SetFigFont{7}{8.4}{\familydefault}{\mddefault}{\updefault}{\color[rgb]{0,0,0}odd}%
}}}}
\put(7366,-6063){\makebox(0,0)[lb]{\smash{{\SetFigFont{7}{8.4}{\familydefault}{\mddefault}{\updefault}{\color[rgb]{0,0,0}even}%
}}}}
\put(11176,-6884){\makebox(0,0)[rb]{\smash{{\SetFigFont{7}{8.4}{\familydefault}{\mddefault}{\updefault}{\color[rgb]{0,0,0}odd}%
}}}}
\put(10414,-6708){\makebox(0,0)[rb]{\smash{{\SetFigFont{7}{8.4}{\familydefault}{\mddefault}{\updefault}{\color[rgb]{0,0,0}odd}%
}}}}
\put(11176,-6063){\makebox(0,0)[rb]{\smash{{\SetFigFont{7}{8.4}{\familydefault}{\mddefault}{\updefault}{\color[rgb]{0,0,0}even}%
}}}}
\put(10473,-7235){\makebox(0,0)[rb]{\smash{{\SetFigFont{7}{8.4}{\familydefault}{\mddefault}{\updefault}{\color[rgb]{0,0,0}even}%
}}}}
\put(9828,-6884){\makebox(0,0)[rb]{\smash{{\SetFigFont{7}{8.4}{\familydefault}{\mddefault}{\updefault}{\color[rgb]{0,0,0}even}%
}}}}
\put(9124,-6708){\makebox(0,0)[rb]{\smash{{\SetFigFont{7}{8.4}{\familydefault}{\mddefault}{\updefault}{\color[rgb]{0,0,0}even}%
}}}}
\put(9183,-7235){\makebox(0,0)[rb]{\smash{{\SetFigFont{7}{8.4}{\familydefault}{\mddefault}{\updefault}{\color[rgb]{0,0,0}even}%
}}}}
\put(9066,-5887){\makebox(0,0)[rb]{\smash{{\SetFigFont{7}{8.4}{\familydefault}{\mddefault}{\updefault}{\color[rgb]{0,0,0}odd}%
}}}}
\put(9124,-6415){\makebox(0,0)[rb]{\smash{{\SetFigFont{7}{8.4}{\familydefault}{\mddefault}{\updefault}{\color[rgb]{0,0,0}odd}%
}}}}
\put(9886,-6063){\makebox(0,0)[rb]{\smash{{\SetFigFont{7}{8.4}{\familydefault}{\mddefault}{\updefault}{\color[rgb]{0,0,0}even}%
}}}}
\put(6428,-6532){\makebox(0,0)[lb]{\smash{{\SetFigFont{9}{10.8}{\familydefault}{\mddefault}{\updefault}{\color[rgb]{0,0,0}$1$}%
}}}}
\put(7542,-6532){\makebox(0,0)[lb]{\smash{{\SetFigFont{9}{10.8}{\familydefault}{\mddefault}{\updefault}{\color[rgb]{0,0,0}$2 (1)$}%
}}}}
\put(6428,-7352){\makebox(0,0)[lb]{\smash{{\SetFigFont{9}{10.8}{\familydefault}{\mddefault}{\updefault}{\color[rgb]{0,0,0}$1$}%
}}}}
\put(7542,-7352){\makebox(0,0)[lb]{\smash{{\SetFigFont{9}{10.8}{\familydefault}{\mddefault}{\updefault}{\color[rgb]{0,0,0}$2(1)$}%
}}}}
\put(9359,-7352){\makebox(0,0)[lb]{\smash{{\SetFigFont{9}{10.8}{\familydefault}{\mddefault}{\updefault}{\color[rgb]{0,0,0}$4$}%
}}}}
\put(9359,-6532){\makebox(0,0)[lb]{\smash{{\SetFigFont{9}{10.8}{\familydefault}{\mddefault}{\updefault}{\color[rgb]{0,0,0}$1$}%
}}}}
\put(10707,-6532){\makebox(0,0)[lb]{\smash{{\SetFigFont{9}{10.8}{\familydefault}{\mddefault}{\updefault}{\color[rgb]{0,0,0}$k$}%
}}}}
\put(10707,-7352){\makebox(0,0)[lb]{\smash{{\SetFigFont{9}{10.8}{\familydefault}{\mddefault}{\updefault}{\color[rgb]{0,0,0}$2$}%
}}}}
\put(10276,-5911){\makebox(0,0)[lb]{\smash{{\SetFigFont{8}{9.6}{\familydefault}{\mddefault}{\updefault}{\color[rgb]{0,0,0}$k$}%
}}}}
\put(10351,-6361){\makebox(0,0)[lb]{\smash{{\SetFigFont{8}{9.6}{\familydefault}{\mddefault}{\updefault}{\color[rgb]{0,0,0}$k$}%
}}}}
\end{picture}%

%% file: figs/wallex1.pstex_t
\begin{picture}(0,0)%
\includegraphics{wallex1.ps}%
\end{picture}%
%
%
\setlength{\unitlength}{3947sp}%
\begingroup\makeatletter\ifx\SetFigFont\undefined%
\gdef\SetFigFont#1#2#3#4#5{%
  \reset@font\fontsize{#1}{#2pt}%
  \fontfamily{#3}\fontseries{#4}\fontshape{#5}%
  \selectfont}%
\fi\endgroup%
\begin{picture}(2052,1206)(8911,-9880)
\put(10051,-9811){\makebox(0,0)[b]{\smash{{\SetFigFont{11}{13.2}{\rmdefault}{\mddefault}{\updefault}{\color[rgb]{0,0,0}$T_1$}%
}}}}
\put(8926,-9586){\makebox(0,0)[rb]{\smash{{\SetFigFont{11}{13.2}{\rmdefault}{\mddefault}{\updefault}{\color[rgb]{0,0,0}$\mu_2$}%
}}}}
\put(8926,-8836){\makebox(0,0)[rb]{\smash{{\SetFigFont{11}{13.2}{\rmdefault}{\mddefault}{\updefault}{\color[rgb]{0,0,0}$\mu_1$}%
}}}}
\put(10990,-8821){\makebox(0,0)[lb]{\smash{{\SetFigFont{11}{13.2}{\rmdefault}{\mddefault}{\updefault}{\color[rgb]{0,0,0}$\nu_1$}%
}}}}
\put(10976,-9114){\makebox(0,0)[lb]{\smash{{\SetFigFont{11}{13.2}{\rmdefault}{\mddefault}{\updefault}{\color[rgb]{0,0,0}$\nu_2$}%
}}}}
\put(10954,-9520){\makebox(0,0)[lb]{\smash{{\SetFigFont{11}{13.2}{\rmdefault}{\mddefault}{\updefault}{\color[rgb]{0,0,0}$\nu_3$}%
}}}}
\end{picture}%

%% file: figs/wallex2.pstex_t
\begin{picture}(0,0)%
\includegraphics{wallex2.ps}%
\end{picture}%
%
%
\setlength{\unitlength}{3947sp}%
\begingroup\makeatletter\ifx\SetFigFont\undefined%
\gdef\SetFigFont#1#2#3#4#5{%
  \reset@font\fontsize{#1}{#2pt}%
  \fontfamily{#3}\fontseries{#4}\fontshape{#5}%
  \selectfont}%
\fi\endgroup%
\begin{picture}(2430,1206)(5461,-9880)
\put(5476,-8836){\makebox(0,0)[rb]{\smash{{\SetFigFont{11}{13.2}{\rmdefault}{\mddefault}{\updefault}{\color[rgb]{0,0,0}$\mu_1$}%
}}}}
\put(6676,-9811){\makebox(0,0)[b]{\smash{{\SetFigFont{11}{13.2}{\rmdefault}{\mddefault}{\updefault}{\color[rgb]{0,0,0}$T_2$}%
}}}}
\put(7876,-8836){\makebox(0,0)[lb]{\smash{{\SetFigFont{11}{13.2}{\rmdefault}{\mddefault}{\updefault}{\color[rgb]{0,0,0}$\nu_1$}%
}}}}
\put(7726,-9586){\makebox(0,0)[lb]{\smash{{\SetFigFont{11}{13.2}{\rmdefault}{\mddefault}{\updefault}{\color[rgb]{0,0,0}$\nu_3$}%
}}}}
\put(7801,-9136){\makebox(0,0)[lb]{\smash{{\SetFigFont{11}{13.2}{\rmdefault}{\mddefault}{\updefault}{\color[rgb]{0,0,0}$\nu_2$}%
}}}}
\put(5476,-9586){\makebox(0,0)[rb]{\smash{{\SetFigFont{11}{13.2}{\rmdefault}{\mddefault}{\updefault}{\color[rgb]{0,0,0}$\mu_2$}%
}}}}
\end{picture}%